\documentclass[11pt,reqno]{amsart} 
\usepackage[a4paper,left=24mm,right=24mm,top=32mm,bottom=30mm]{geometry} 
\usepackage{verbatim,amsmath,amssymb,enumitem} 
\usepackage{graphicx} 
\usepackage{color} 
\usepackage{esint}
\usepackage[utf8x]{inputenc}
\usepackage{hyperref} 
\usepackage{enumitem}
\usepackage{calc}

\newcommand{\N}{\mathbb{N}}
\newcommand{\R}{\mathbb{R}}

\newcommand{\eps}{\epsilon}
\newcommand{\p}{\partial}
\newcommand{\bq}{\begin{eqnarray}}
\newcommand{\eq}{\end{eqnarray}}
\newcommand{\ee}{\end{align*}}

\newcommand{\un}{u_{\mathrm{new}}}

\theoremstyle{plain}
\newtheorem{theorem}{Theorem}
\newtheorem{lemma}[theorem]{Lemma}
\newtheorem{proposition}[theorem]{Proposition}

\theoremstyle{definition}

\newtheorem*{notation*}{Notation}

\theoremstyle{remark}
\newtheorem{remark}[theorem]{Remark}

\begin{document} 

\title[Active transport in biological cells]% short title
{Global existence for a bulk/surface model for active-transport-induced polarisation in biological cells}% full title

\author[K.~Anguige]{Keith Anguige}
\address{Abteilung für Angewandte Mathematik,
	Albert-Ludwigs-Universität Freiburg,
	Hermann-Herder-Str. 10,
	79104 Freiburg i. Br., Germany}
\email{keith.anguige@mathematik.uni-freiburg.de} 

\author[M.~R\"oger]{Matthias R\"oger}
 \address{Fakult\"at f\"ur Mathematik, Technische Universit\"at Dortmund\\Vogelpothsweg 87,
 D-44227 Dortmund, Germany\\
 matthias.roeger@tu-dortmund.de}
 \email{matthias.roeger@tu-dortmund.de} 

\subjclass[2010]{35Q92,35B44,35K57}

 \keywords{Partial differential equations on surfaces, coupled bulk/surface processes, cell polarisation, active transport, blow-up}

\begin{abstract}
We consider a coupled bulk/surface model for advection and diffusion of interacting chemical species in biological cells. Specifically, we consider a signalling protein that can exist in both a cytosolic and a membrane-bound state, along with a variable that gives a coarse-grained description of the cytoskeleton. The main focus of our work is on the well-posedness of the model, whereby the coupling at the boundary is the main source of analytical difficulty. {\em A priori} $L^p$-estimates, together with classical Schauder theory, deliver global existence of classical solutions for small data on bounded, Lipschitz domains. For two physically reasonable regularised versions of the boundary coupling, we are able to prove global existence of solutions for {\em arbitrary} data. In addition, we prove the existence of a family of steady-state solutions of the main model which are parametrised by the total mass of the membrane-bound signal molecule.
\end{abstract}
\date{\today}
\thanks{The second author was partially funded by DFG under contract RO 3850/2-1, and thanks Juan Vel{\'a}zquez for interesting discussions.}

\maketitle
%=======================
%
%=======================
\section{Introduction}
\label{intro}
Cell-polarisation processes are the key to many biological functions, such as cell movement, differentiation and communication \cite{Nels03}. A prominent example is the budding of yeast, with the Rho GTPase protein Cdc42 as the main polarity marker \cite{CaSn02}: preceding any mechanical deformation of the cell, one sees the emergence and maintenance of an inhomogeneous distribution of regulatory proteins at the cell membrane and in the inner cytosolic domain \cite{PaBi07}. Various mechanisms have been identified that can contribute to this kind of symmetry breaking, whereby a distinction is commonly made between \emph{driven} and \emph{spontaneous} cell polarisation \cite{HBPV09}. The former is induced by either extracellular chemical gradients (chemotaxis) or historical markers at the membrane, while the latter is, instead, a consequence of interactions between different regulatory proteins and/or other constituents of the cell. Turing-type interactions between short-range activators and long-range inhibitors are one scenario that may lead to spontaneous polarisation \cite{GoPo08,RaeRoe14}. A distinct and well-documented mechanism  \cite{WAWL03,WWSL04,LiGu08} is a positive feedback between membrane-recruited signaling proteins and the cytoskeleton. The latter is built from actin monomers that polymerise to form long filaments: activated Cdc42 directs actin polymerisation at the membrane, and, in turn, Cdc42 is actively transported along the cytoskeleton filaments towards the membrane, leading to a positive feedback loop \cite{OnRa09}. Such active transport relies on a permanent energy input from ATP hydrolysis, and hence constitutes an example of an out-of-equilibrium system.

In this paper we will analyse a mathematical model for the kind of actin-mediated spontaneous cell polarisation just described. The basis of our work is a modification of a model introduced by Hawkins et al.~\cite{HBPV09} that uses a coarse-grained description of the actin-filaments, and that leads to a coupled bulk/surface reaction--diffusion--advection model with a nonlinearity in the bulk of chemotaxis type. Our mathematical analysis will demonstrate the well-posedness of the associated initial-value problem under a smallness-condition on the initial data.

In order to introduce the model, consider a domain $B\subset\mathbb{R}^3$ with boundary $\Gamma=\partial B$, modelling the cell and its outer cell-membrane. Consider further three chemical species (resp. their concentrations): the cytosolic concentration $V$ of the biochemical messenger, the concentration $u$ of the membrane-bound messenger, and the concentration of actin filaments in the cell, $c$.  While $V$ and $c$ live in the cell interior, $B$, the function $u$ has its domain of definition on the cell membrane, $\Gamma$. We assume that $V$ can diffuse in the cell interior, and that $V$ and $u$ are exchanged at the membrane. Moreover, $u$ can diffuse throughout the membrane (i.e., tangentially), and acts as a boundary source for the (very) diffusible species $c$, generating an advective velocity field, $\nabla c$, in $B$ which tends to transport $V$ towards the boundary. These model assumptions can be summarised as the following coupled system: consider a given time interval $(0,T)$, and functions $V,c:B\times [0,T)\to\R$, $u:\Gamma\times [0,T)\to\R$, such that
\bq
\p_tV & = & D\Delta V - \nabla\cdot(V\nabla c),\label{V}\\
0  & = & \Delta c - \alpha c \label{c}
\eq
in $B\times (0,T)$, subject to the flux conditions
\begin{align}
- \nu\cdot(D\nabla V - V\nabla c) &= q(V,u),\label{Vflux}\\
%\end{align}
%and
%\begin{align}
	\nu\cdot\nabla c &= \beta u\label{cflux}
\end{align} 
on $\Gamma\times (0,T)$, and such that $u$ satisfies
\begin{align}
\p_tu  =  d\Delta_{\Gamma} u + q(V,u)  \label{u}
\end{align}
on $\Gamma\times (0,T)$.

Here $\Delta_{\Gamma}$ is the Laplace-Beltrami operator on the manifold $\Gamma = \p B$, and we are assuming $\alpha>0, d>0$ and $D>0$. 
To close the system, we prescribe a constitutive law for the bulk--surface exchange term. In the present paper we will mainly consider a simple linear law
\begin{align}
	q(V,u) \,:=\, k_1V  - k_2u, \label{eq:def-q}
\end{align}
although two alternative (regularised) choices for $q(V, u)$, resp. the Neumann source in \eqref{cflux}, will briefly be examined towards the end of the paper, in Section \ref{reg}. The system \eqref{V}-\eqref{eq:def-q} is complemented by the initial conditions
\begin{align}
	V(\cdot,0)\,=\, V_0,\quad u(\cdot,0)\,=\, u_0, \label{eq:initial}
\end{align}
where $V_0:B\to \R$ and $u_0:\Gamma\to\R$ are given data that we assume to be smooth, bounded, integrable in space, and nonnegative.

Note that, by construction (and an integration by parts), the system (\ref{V})-(\ref{u}) conserves the total mass of $V$ and $u$, regardless of the choice of $q(V,u)$:
\begin{align}
	\int_B V(x,t)~dx + \int_{\Gamma} u(y,t)~dS(y) = \int_B V_0(x)~dx + \int_{\Gamma} u_0(y)~dS(y) =: M  \quad\text{ for all }t\geq 0. \label{eq:mass}
\end{align}

The system \eqref{V}-\eqref{eq:def-q} is a generalisation of the model introduced in \cite{HBPV09} and further analysed in \cite{CHMV12,CMMV13,MPCV16}. We will discuss below the results of this analysis in the context of our own findings.

Our system is somewhat similar to the celebrated Patlak--Keller--Segel (PKS) chemotaxis system  \cite{Patl53,Kese70} (see also the review \cite{Hors03}), which is given by
\begin{align*}
	\partial_tV\,=\, \nabla\cdot(\nabla V -\chi V\nabla c),\quad \partial_t c \,=\, \Delta c -\alpha c + V\qquad\text{ in }B\times (0,T),
\end{align*}
supplemented by initial and boundary conditions. 

A typical feature of PKS-type models is the existence of a critical space (in many cases $L^{\frac{n}{2}}$, where $n$ is the space dimension) and a threshold phenomenon: for sub-critical inital data one has global existence of solutions, whereas for data sufficiently large (in the sense of the critical norm), solutions blow-up in finite time. This phenomenon was first observed by J\"ager and Luckhaus \cite{JaLu92}, and is by now rather well-understood -- see, in particular, \cite{HeVe96,Vela04} and the review \cite{Hors03}. 

Compared to our model, the classical PKS model exhibits a much more direct feedback between advection of $V$ up the concentration-gradient of $c$ and the up-regulation of $c$ by $V$. More specifically, in \eqref{V}-\eqref{eq:def-q} the influx of $c$ is generated at the boundary by a source that is proportional to the concentration of membrane-bound messenger, which in turn is positively regulated by $V$ at the boundary. This shows a higher level of amplification and nonlocality in the feedback mechanism, and it suggests that any blow-up should occur at the boundary.

Our more indirect and nonlocal feedback, mediated by a diffusible messenger on $\Gamma$, ostensibly represents some kind of regularisation of PKS dynamics, but unfortunately makes the analysis rather more involved. In particular, it is not obvious that there is any Lyapunov functional present. We therefore have to do without the very efficient tools available in the analysis of the PKS model \cite{SeSo01,BlDP06}, relying instead on J\"ager--Luckhaus--type techniques \cite{JaLu92} of estimating $L^p$-norms, thus obtaining (conditional on the size of the initial data) {\em a-priori} bounds via Gagliardo--Nirenberg--Sobolev type estimates.  Whether or not our indirect feedback mechanism is so weak as to prevent blow-up of solutions turns out to be a rather difficult question to answer, and, in particular, we do not expect the kind of clear-cut threshold characterisation that is now available for the PKS model. %We are, however, at least able to give some partial answers.

Our first main result states that for both $B=\R^3_+$, the upper half space, and $B$ a smooth, bounded domain, a certain smallness condition on the initial data guarantees {\em a-priori} bounds for classical solutions; see Proposition \ref{prop} and Proposition \ref{prop_bdd_dom_Q_p}. The smallness condition is given in terms of the total mass $M$ and integrals of $p$-th powers, specifically the quantities
\begin{align*}
	\vartheta(p)\,:=\,& \int_B V^p_0\,dx + c_1\int_\Gamma  u^p_0\,dS(y) \qquad\text{ if }p\geq 2,\\
	\vartheta(p)\,:=\,& \int_B V^p_0\,dx + c_1\int_\Gamma  u^p_0\,dS(y) + \int_{\Gamma}u^2_0\,dS(y) \qquad\text{ if } 1 < p < 2.
\end{align*}
After this, we focus entirely on the case of bounded domains, $B$, and first show in Section \ref{globex3} that simultaneous explosion of both the $L^4(B)$-norm of $V$ and the $L^4(\Gamma)$-norm of $u$ is necessary for blow-up of classical solutions. Then we prove several results which eventually give unique existence of classical solutions in certain parabolic Hölder spaces, provided some $\vartheta(p)$ with $p>1$ is sufficiently small; see Theorem \ref{thm2}. 

The results mentioned up to this point all relate to the linear exchange law \eqref{eq:def-q} for $q$, and the linear flux condition \eqref{cflux}. The analysis is extended in Section \ref{reg}, where we show that if any of these constitutive laws is replaced by a suitable nonlinear relation that gives an {\em a-priori} bound on the respective flux, then globally existing classical solutions are guaranteed for arbitrary data.
 
Our final result, Theorem \ref{arb_steady}, is to prove the existence of a continuum of steady states with small $u$-mass. 

To add a little more context to our results, note that the system \eqref{V}-\eqref{eq:def-q} is a slight generalisation of the model introduced in \cite{HBPV09}. There, a reduction to two space dimensions, the choice $\alpha=0$, and a specific cell shape are considered. A linear-stability analysis of spatially homogeneous steady states is presented which reveals the possibility of spontaneous polarisation. Numerical simulations in \cite{CMMV13} and, for an extended model, in \cite{MPCV16} show the development of steep concentration gradients.
%This is characterised by an instability under spatially heterogeneous perturbations with maximally unstable wave-length. 
In \cite{CHMV12}, both a 
1-d reduction of the model from \cite{HBPV09} and, for arbitrary space dimensions, a two-variable reduction in the upper half-space are analysed, such that the membrane-bound messenger concentration is taken as the trace of $V$, rather than being an independent variable. A thorough mathematical analysis is presented, such that the model is shown to behave similarly to the PKS system, with, however, a different critical space: smallness of the initial data in $L^n$ implies global existence of weak solutions, while under suitable conditions on the initial data, %(monotonic spatial decay and small second moments for the $V$ variable, and a further condition on mixed derivatives of $\log V$) 
and additional assumptions on the solution, finite-time blow-up occurs. %Numerical simulations \cite{CMMV13} demonstrate the possibility of cell polarisation. In a very recent work \cite{MPCV16}, an extended model of the form \eqref{V}-\eqref{eq:def-q} is considered in a 2-d annulus (modeling a spherical cell with a nucleus), where the in-flux law \eqref{cflux} includes a factor of proportionality determined by an external stimulus. Numerical simulations show that the model allows for correct quantitative predictions.

Compared with the results in \cite{CHMV12}, we observe that the critical norm for $V$ in our model is between $L^1$ and any $L^p$, $p>1$, in contrast to the critical exponent $p=n=3$ in \cite{CHMV12}. In addition, in our case a smallness condition on the additional surface variable $u$ is necessary, $L^2(\Gamma)$ being critical here. Under these smallness conditions, we are able to give a rather satisfactory result on global existence. As already mentioned, in \cite{CHMV12} it was possible to give criteria for the failure of global existence, while here we have to leave this question open. The main obstacle is that the usual moment-estimate methods for proving blow-up in the case of (in some sense) concentrated data on a half-space seem to be very difficult to apply in our model.

\begin{notation*}
In the rest of the paper we will often simplify the notation by dropping the symbol $dx$ for integration over $B$, resp. $dS(y)$ for integration over the surface $\Gamma=\partial B$.  For $B$ bounded we denote by $|B|=\mathcal L^3(B)$ the volume of $B$ and by $|\Gamma|=\mathcal H^2(\Gamma)$ the surface area of the cell boundary.

Again for simplicity, we will just write $\nabla$ and $\Delta$ for the surface gradient $\nabla_\Gamma$ and Laplace--Beltrami operator $\Delta_\Gamma$ on $\Gamma$.

We let $B_T = \overline{B}\times [0,T]$, $\Gamma_T = \Gamma\times [0,T]$. By $W^{2,1}_p(B_T)$ we denote the parabolic Sobolev space of functions with one time and two space-derivatives in $L^p(B_T)$. Furthermore, we use the standard Hölder spaces $C^{k,\alpha}(\bar B)$, $k\in\N_0$, $0\leq\alpha<1$ with norms $\|\cdot\|_{C^{k,\alpha}(\bar B)}$, and the parabolic Hölder spaces $H^{l,l/2}(B_T)$, $H^{l,l/2}(\Gamma_T)$, $l>0$, with norms $|\cdot|^{(l)}$ (see \cite{lady2} for definition and properties).
\end{notation*}

%=======================
%
%=======================
\section{A-priori estimates for small data}\label{globex}
Throughout this section we assume that $(V,c,u)$ is a smooth, non-negative solution of \eqref{V}-\eqref{eq:initial} on a time interval $[0,T)$. Note that classical solutions corresponding to non-negative data on a bounded domain are automatically non-negative on their interval of existence - moreover, positivity is also preserved for smooth-enough solutions in a half-space (see the Appendix for a proof of these statements). Furthermore, by \eqref{eq:mass}, the total mass is conserved, and hence the $L^1(B)$ norm of $V(\cdot,t)$ and the $L^1(\Gamma)$ norm of $u(\cdot,t)$ are uniformly bounded in time. 

Our immediate aim is to obtain global {\em a priori} $L^p$-estimates for $V$ and $u$, given small data, which will preclude the concentration of mass, either in the interior or on the boundary. First we derive an estimate for the evolution of spatial integrals of $p$-th powers of $u$ and $V$.
\begin{lemma}\label{lem:p-powers}
For any $p>1$ we have
\bq
\frac{d}{dt} \int_B V^p & \leq & -\frac{4(p-1)}{p}D\int_B |\nabla V^{\frac{p}{2}}|^2 + \int_\Gamma p(-k_1 V^p +k_2V^{p-1}u) + \int_{\Gamma}(p-1)\beta uV^p, \label{eq:Vp-diff}\\
\frac{d}{dt} \int_\Gamma u^p & = &   -\frac{4(p-1)}{p}d\int_\Gamma |\nabla u^{\frac{p}{2}}|^2 + p\int_\Gamma (k_1 Vu^{p-1} -k_2u^p) \label{eq:up-diff}
\eq
and, with $c_1=(k_2/k_1)^{p-1}$, $D_p= \frac{4(p-1)}{p}D$, $d_p=\frac{4c_1(p-1)}{p}d$,
\bq
\frac{d}{dt}\left(\int_{\Gamma} c_1 u^p  + \int_B V^p\right) & \leq &
-D_p\int_B |\nabla V^{\frac{p}{2}}|^2 - d_p\int_{\Gamma}|\nabla_{\Gamma} u^{\frac{p}{2}}|^2
+ \beta (p-1)\int_{\Gamma}V^pu,\label{est1}
\eq
\end{lemma}
\begin{proof}
First we test (\ref{V}) with $V^{p-1}$, and integrate by parts, to obtain
\begin{eqnarray}
\frac{d}{dt}\left(\int_B V^p\right) & = & -\frac{4(p-1)}{p}D\int_B |\nabla V^{\frac{p}{2}}|^2 - \int_{\Gamma}p(k_1V - k_2u)V^{p-1} + p\int_B V\nabla V^{p-1}\cdot\nabla c\nonumber\\ 
   & \leq & -\frac{4(p-1)}{p}D\int_B |\nabla V^{\frac{p}{2}}|^2 - \int_{\Gamma}p(k_1V - k_2u)V^{p-1} +  \beta (p-1)\int_{\Gamma}V^pu,\label{est}
\end{eqnarray}
where we used (\ref{c}), (\ref{cflux}), $c\geq 0$ and
\begin{eqnarray}
p\int_B V\nabla V^{p-1}\cdot\nabla c & = & (p-1)\int_B\nabla V^p\cdot\nabla c\nonumber\\
    & = & (p-1)\left(-\int_B V^p\Delta c + \int_{\Gamma}V^p\nu\cdot\nabla c \right)\nonumber\\
    & \leq & (p-1)\beta\int_{\Gamma}V^p u, \label{nonlinear}
\end{eqnarray}
to get the last line in (\ref{est}). This establishes (\ref{eq:Vp-diff}).

Next, (\ref{eq:up-diff}) follows easily by testing (\ref{u}) with $u^{p-1}$ and integrating the Laplacian term by parts. Finally, (\ref{est1}) follows by taking the appropriate weighted sum of (\ref{eq:Vp-diff}) and (\ref{eq:up-diff}), and noting that $-(x-y)(x^{p-1}-y^{p-1})\leq 0$ for all $x,y\in\mathbb{R}^{+}_0$, $p\geq 1$.
\end{proof}

\subsection{Unbounded domains: the half-space}
We first restrict ourselves to the case $B=\R^3_+\,=\, \{x\in\R^3\,:\, x_3>0\}$, $\Gamma = \R^2\times\{0\}\mathrel{\widehat=} \R^2$, and assume that $u$ and $V$ are bounded and integrable in space for all times. We then obtain {\em a-priori} bounds for any $L^p$ norm, provided the initial data are sufficiently small.
 
In the following we set
\begin{align}
Q_p(t) \,:=\, 
\left\{
\begin{array}{lcl}
 \int_B V^p(x,t)\,dx + \int_\Gamma  c_1u^p(y,t)\,dS(y) & : & p\geq 2\\
 & &\\
 \int_B V^p(x,t)\,dx + \int_\Gamma  c_1u^p(y,t)\,dS(y) + \int_{\Gamma}u^2\,dS(y) & : & 1 < p < 2
\end{array}
\right.\label{eq:def-qp}
\end{align}
with $c_1=(k_2/k_1)^{p-1}$, as in Lemma \ref{lem:p-powers}.

\begin{proposition}\label{prop}
We have the following properties:
\begin{enumerate}%[labelindent=0pt,labelwidth=\widthof{\ref{last-item}},itemindent=1.58em,leftmargin=!]
\item 
For any $r>1$, there exist $\delta_1,\delta_2>0$, depending only on the model parameters such that the condition
\begin{align}
	M\,<\,\delta_1\quad\text{ and }\quad Q_r(0)\,\leq\,\delta_2,  \label{eq:cond-r}
\end{align}
implies that
\begin{align}
	&\sup_{0\leq t<T} Q_r(t) \leq C(\delta_1,\delta_2)<\infty, \label{eq:a15}
\end{align}
where $C$ is a modulus of continuity in its second argument.
\item 
For all $p\geq 1$ there exists a constant $M_p>0$ such that if \eqref{eq:cond-r} holds for some $r\in(1,\infty)$ and $M<M_p$ then
\begin{align}
	&\sup_{0\leq t<T} Q_p(t) \leq C(M_p,Q_p(0))<\infty, \label{eq:a17}
\end{align}
where $C$ is a modulus of continuity in its second argument.
\end{enumerate}
\end{proposition}

\begin{proof}
Let $r>1$ be given, and consider the following trace inequality for half-spaces (see, for example, \cite{Naza06}), which holds for any $f\in H^1(B)$:
\begin{align}
	\|f\|_{L^4(\Gamma)}\,\leq\, C\|\nabla f\|_{L^2(B)}. \label{eq:trace}
\end{align}
We therefore deduce, for any $p>1$,
\begin{align*}
	\int_\Gamma uV^p \,\leq\, \|u\|_{L^2(\Gamma)}\|V^{\frac{p}{2}}\|_{L^4(\Gamma)}^2 \,\leq\, C\|u\|_{L^2(\Gamma)}\|\nabla V^{\frac{p}{2}}\|_{L^2(B)}^2. %\label{eq:trace_uVp}
\end{align*}
In particular, using this inequality for $p=r$, we deduce from \eqref{est1} that
\begin{align}
	\frac{d}{dt}\left(\int_{\Gamma} c_1 u^r  + \int_B V^r\right) \,&\leq\,-D_r \int_B |\nabla V^{\frac{r}{2}}|^2 + \beta(r-1)\int_\Gamma V^r u \notag\\
	&\leq\, -D_r \int_B |\nabla V^{\frac{r}{2}}|^2 + C\beta(r-1)\|V^{\frac{r}{2}}\|_{L^4(\Gamma)}^2\| u\|_{L^2(\Gamma)} \notag\\
	&\leq\, 	-(D_r - C\beta(r-1)\|u\|_{L^2(\Gamma)})\int_B |\nabla V^{\frac{r}{2}}|^2. \label{eq:prop2-1}
\end{align}
In the case $r\geq 2$, we see by interpolation that
\begin{align*}
	\|u\|_{L^2(\Gamma)} \,\leq\, \|u\|_{L^1(\Gamma)}^{\frac{r-2}{2(r-1)}}\|u\|_{L^r(\Gamma)}^{\frac{r}{2(r-1)}}\,\leq\, M^{\frac{r-2}{2(r-1)}}\|u\|_{L^r(\Gamma)}^{\frac{r}{2(r-1)}}.
\end{align*}
Using this in the above inequality yields
\begin{align}
	\frac{d}{dt}\left(\int_{\Gamma} c_1 u^r  + \int_B V^r\right) \,&\leq\, 	-\Big(D_r - C\beta(r-1)M^{\frac{r-2}{2(r-1)}}\|u\|_{L^r(\Gamma)}^{\frac{r}{2(r-1)}}\Big)\int_B |\nabla V^{\frac{r}{2}}|^2. \label{eq:prop2-2}
\end{align}
In particular, if $\delta_1$, $\delta_2$ are chosen so small that
\begin{align*}
	D_r - C\beta(r-1)M^{\frac{r-2}{2(r-1)}}(c_1^{-1}\delta_2)^{\frac{1}{2(r-1)}}\,\geq\,0
\end{align*}
when $M<\delta_1$, we have for all times that
\begin{align}
	\int_{\Gamma} c_1 u^r  + \int_B V^r \,&\leq\, \delta_2, \label{eq:prop2-3}
\end{align}
provided this condition holds at $t=0$. This proves \eqref{eq:a15} for $r\geq 2$. 

We continue to assume $r\geq 2$, and next consider an arbitrary $p\geq 1$. As in \eqref{eq:prop2-1}, \eqref{eq:prop2-2} we infer, using (\ref{eq:prop2-3}) that
\begin{align*}
	\frac{d}{dt}\left(\int_{\Gamma} c_1 u^p  + \int_B V^p\right) 
	\,&\leq\, -\Big(D_p - C\beta(p-1)M^{\frac{r-2}{2(r-1)}}\|u\|_{L^r(\Gamma)}^{\frac{r}{2(r-1)}}\Big)\int_B |\nabla V^{\frac{p}{2}}|^2 \notag\\
	\,&\leq\, -\Big(D_p - C\beta(p-1)M^{\frac{r-2}{2(r-1)}}(c_1^{-1}\delta_2)^{\frac{1}{2(r-1)}}\Big)\int_B |\nabla V^{\frac{p}{2}}|^2. %\label{eq:prop2-4}
\end{align*}
We therefore obtain that the right-hand side is non-positive for all times if $M<M_p$ for $M_p$ chosen sufficiently small. This proves \eqref{eq:a17} for $r\geq 2, p\geq 2$. By the interpolation inequality
\begin{align*}
	\|V\|_{L^p(B)} \,\leq\, \|V\|_{L^1(B)}^\theta\|V\|_{L^r(B)}^{1-\theta},\quad\theta=\frac{r-p}{rp-p},
\end{align*}
the corresponding inequalities for $\|u\|_{L^p(\Gamma)}, \|u\|_{L^2(\Gamma)}$, and by  \eqref{eq:a15} we obtain that \eqref{eq:a17} also holds for $r\geq 2, p< 2$.

We now turn to the case $1<r<2$. From \eqref{eq:up-diff} we have
\begin{align}
	\frac{d}{dt} \int_\Gamma u^2 \,&\leq\,   -2d\int_\Gamma |\nabla u|^2 + 2k_1\int_\Gamma Vu -c_1u^2.\label{eq:up-diff2}
\end{align}
Observe that $Vu -c_1u^2\leq 0$ if $V\leq c_1u$, that $Vu-c_1u^2\leq V^ru$ if $V\geq 1$, and that $Vu -c_1u^2 \,\leq\, c_1^{1-r}V^ru^{2-r}$ holds in the set $\{V>c_1u\}\cap\{V<1\}$. Hence,
\begin{align}
	Vu -c_1u^2 \,\leq\, V^ru + c_1^{1-r}V^ru^{2-r}.\label{eq:Vu}
\end{align}
Moreover, again by \eqref{eq:trace}, we can argue as above to get 
\begin{align*}
	\int_\Gamma V^ru \,\leq\, C_r\|u\|_{L^2(\Gamma)}\int_B |\nabla V^{\frac{r}{2}}|^2 %\label{eq:product_est1}
\end{align*}
and 
\begin{align*}
	\int_\Gamma V^ru^{2-r} \,\leq\, \|u\|_{L^{4-2r}(\Gamma)}^{2-r}\|V^{\frac{r}{2}}\|_{L^4(\Gamma)}^2\,\leq\, C\|u\|_{L^{4-2r}(\Gamma)}^{2-r} \int_B |\nabla V^{\frac{r}{2}}|^2. %\label{eq:product_est2}
\end{align*}
Thus, for $1<r\leq\frac{3}{2}$, (\ref{eq:prop2-1}) and the interpolation inequality $\|u\|_{L^{4-2r}(\Gamma)}\leq\|u\|_{L^1(\Gamma)}^{1-\theta}\|u\|_{L^2(\Gamma)}^{\theta}$, $\theta = \frac{3-2r}{2-r}$, lead to
\begin{align*}%\label{eq:small_r_est}
\frac{d}{dt}\left(\int_{\Gamma}c_1 u^r + \int_B V^r + \int_{\Gamma} u^2\right)\leq -\left(D_r - C\|u\|_{L^2(\Gamma)} - CM^{r-1}\|u\|_{L^2(\Gamma)}^{3-2r}\right)\int_B |\nabla V^{\frac{r}{2}}|^2,
\end{align*}
which proves \eqref{eq:a15} for $1<r\leq\frac{3}{2}$.

Finally, in order to close the gap between $r=\frac{3}{2}$ and $r=2$, we argue as follows. Suppose $r\in \left(\frac{3}{2},2\right)$. Then, for $s\in (1,\frac{3}{2}]$, interpolation enables us to control $Q_s(0)$ in terms of $Q_r(0)$ and $M$. In particular, by taking $Q_r(0)$ and $M$ sufficiently small,
%, and recalling (\ref{eq:small_r_est}) (with $s$ replacing $r$) and (\ref{eq:prop2-1}), it can be arranged that $\|u\|_{L^2(\Gamma)}(t)\leq D_r/C\beta(p-1)$ (and much smaller if required), and also that 
 \eqref{eq:a15} holds with $r$ replaced by $s$. In particular, $\|u\|_{L^2(\Gamma)}(t)\leq D_r/C\beta(p-1)$ for $M_r>0$ sufficiently small. Then \eqref{eq:prop2-1} implies that
\begin{align*}
\left(\int_{\Gamma}c_1 u^r + \int_B V^r\right)(t)\leq\left(\int_{\Gamma}c_1 u^r + \int_B V^r\right)(0),
\end{align*}
thus proving \eqref{eq:a15} for $r\in \left(\frac{3}{2},2\right)$. 

Finally, in the case $1<r<2$, we can use very similar arguments to show that \eqref{eq:a17} follows from (\ref{eq:a15}) and (\ref{eq:prop2-1}).
\end{proof}

\subsection{Bounded domains}
We next consider the case of a bounded, open set $B\subset\R^3$, such that the smooth boundary $\Gamma=\partial B$ has a finite number of connected components. The estimates used will be similar to those which were useful in the half-space case, but for bounded domains we have to take into account the fact that Gagliardo--Sobolev embeddings of the form \eqref{eq:trace} only hold for functions with mean-value zero, thus making the calculations somewhat more complicated. In the sequel, we will make frequent use of the mean values $v_{\frac{p}{2}} := \frac{1}{|B|}\int_B V^{\frac{p}{2}}$ for $p\geq 1$. 

Our first result is the following. 

\begin{proposition}\label{prop_bdd_dom_Q_p}
Let $B\subset\mathbb{R}^3$ be bounded, and suppose $p\in (1,\infty)$. Then there exist a constant $c_p>0$ and a modulus of continuity $\sigma_p$, such that if $Q_p(0)<c_p$ then
\begin{align}
	\sup_{t>0} Q_p(t)\,&\leq\, \sigma_p(Q_p(0)),  \label{eq:pro3}\\	
	\sup_{t>0} \int_\Gamma u^2(t)\,&\leq\, \sigma_p(Q_p(0)).
	 \label{eq:pro3-2}
\end{align}
\end{proposition}

\begin{proof}
Using \cite[Section 2 (2.27)]{lady1} we obtain
\begin{align}
\int_{\Gamma} uV^p  \leq  2\int_{\Gamma} u\left(V^{\frac{p}{2}} - v_{\frac{p}{2}}\right)^2 + 2v_{\frac{p}{2}}^2\int_{\Gamma} u \leq  C\|u\|_2\int_B|\nabla V^{\frac{p}{2}}|^2 + 2v_{\frac{p}{2}}^2\int_{\Gamma} u.\label{eq:alt_uV_est}
\end{align}

Hence, \eqref{est1} implies
\bq
\frac{d}{dt}\left(\int_B V^p + c_1\int_{\Gamma} u^p\right) & \leq & -D_p \int_B |\nabla V^{\frac{p}{2}}|^2 - d_p\int_{\Gamma} |\nabla u^{\frac{p}{2}}|^2\nonumber\\
 & + & C_p\|u\|_{L^2(\Gamma)}\int_B |\nabla V^{\frac{p}{2}}|^2 + C_p'v_{\frac{p}{2}}^2\int_{\Gamma}u\nonumber\\
 & \leq & -\left(\frac{D_p}{2} - C_p\|u\|_{L^2(\Gamma)}\right)\int_B |\nabla V^{\frac{p}{2}}|^2\nonumber\\
 & - & \frac{D_p}{2}\int_B |\nabla V^{\frac{p}{2}}|^2 - d_p\int_{\Gamma}|\nabla u^{\frac{p}{2}}|^2 +  C_p'v_{\frac{p}{2}}^2\int_{\Gamma}u.\label{eq:alt_Qp_est}
\eq

We will show that, for $c_p>0$ sufficiently small,
\begin{align}
	\|u\|_{L^2(\Gamma)}(t)\,<\,\frac{D_p}{2C_p} \label{eq:u2-small}
\end{align}
holds for all $t\geq 0$. For the moment, we assume that this property holds up to some particular time, $t$.

We deduce from \eqref{eq:alt_Qp_est} that $\int_B V^p + c_1\int_{\Gamma}u^p$ is decreasing, provided
\begin{align}
	C_p'v_{\frac{p}{2}}^2\int_{\Gamma}u\leq\max\left\{\frac{D_p}{2}\int_B |\nabla V^{\frac{p}{2}}|^2, d_p\int_{\Gamma}|\nabla u^{\frac{p}{2}}|^2\right\}. \label{eq:decreasing}
\end{align}
Now consider the case where (\ref{eq:decreasing}) does {\em not} hold, and hence
\begin{align}
	\max\left\{\frac{D_p}{2}\int_B |\nabla V^{\frac{p}{2}}|^2, d_p\int_{\Gamma}|\nabla u^{\frac{p}{2}}|^2\right\}\,<\, C_p'v_{\frac{p}{2}}^2\int_{\Gamma}u. \label{eq:nondecreasing}
\end{align}

Note that
\begin{align*}
	\int_B|V^{\frac{p}{2}} - v_{\frac{p}{2}}|^2  = - v_{\frac{p}{2}}^2|B| + \int_B V^p, \end{align*}
and therefore by the Poincar\'{e}-Wirtinger inequality
\begin{align}
	\int_B V^p \leq  v_{\frac{p}{2}}^2|B| + C_p\int_B |\nabla V^{\frac{p}{2}}|^2\leq v_{\frac{p}{2}}^2\left[2|B| + \frac{2C_pC_p'}{D_p}\int_{\Gamma}u\right]. \label{eq:b30-1}
\end{align}

Moreover, for $p\geq 2$, we obtain by interpolation
\begin{align}
	v_{\frac{p}{2}}^{\frac{2}{p}}\leq\frac{1}{|B|^{\frac{2}{p}}}\|V\|_{L^1(B)}^{\frac{1}{p-1}}\|V\|_{L^p(B)}^{\frac{p-2}{p-1}}, \label{eq:b30-2}
\end{align}
and deduce from \eqref{eq:b30-1} that
\begin{align*} 
	\|V\|_{L^p(B)}^p \,\leq\, C_p\|V\|_{L^1(B)}^{\frac{p}{p-1}}\|V\|_{L^p(B)}^{\frac{p(p-2)}{p-1}}\left[1 + \left(\int_{\Gamma} u \right)\right],
\end{align*}
hence
\begin{align} 
	\|V\|_{L^p(B)}^p \,\leq\, C_p\|V\|_{L^1(B)}^{p}\left[1 + \left(\int_{\Gamma} u \right)^{p-1}\right]\,\leq\, C_pM^p(1-M^{p-1}). \label{eq:rho-p0}
\end{align}
Also, by the Poincar\'{e}-Wirtinger inequality, \eqref{eq:nondecreasing}, \eqref{eq:b30-2} and \eqref{eq:rho-p0},
\begin{align}\label{u_poinc}
	\int_{\Gamma} u^p \leq u_{\frac{p}{2}}^2|\Gamma| + C_p\int_{\Gamma}|\nabla u^{\frac{p}{2}}|^2 \,\leq\, u_{\frac{p}{2}}^2|\Gamma| + Cv_{\frac{p}{2}}^2\int_\Gamma u\,\leq\, u_{\frac{p}{2}}^2|\Gamma| + \rho_p^{(1)}(M),
\end{align}
where $\rho_p^{(1)}$ is a modulus of continuity.
Furthermore, again for $p\geq 2$,
\begin{align*}
	u_{\frac{p}{2}}^{\frac{2}{p}}\leq\frac{1}{|\Gamma|^{\frac{2}{p}}}\|u\|_{L^1(\Gamma)}^{\frac{1}{p-1}}\|u\|_{L^p(\Gamma)}^{\frac{p-2}{p-1}}.
\end{align*}

Thus, by \eqref{u_poinc} and Young's inequality
\begin{align*}
	\int_{\Gamma} u^p \leq C_p M^{\frac{p}{p-1}}\|u\|_{L^p(\Gamma)}^{\frac{p(p-2)}{p-1}} + \rho_p^{(1)}(M)\,\leq\, \frac{1}{2}\int_{\Gamma} u^p + \rho_p^{(2)}(M) +\rho_p^{(1)}(M),
\end{align*}
where $\rho_p^{(2)}$ is another modulus of continuity, and hence 
\begin{align}
	\int_{\Gamma}u^p\,\leq\, \rho_p^{(3)}(M),\quad\text{ for some modulus of continuity }\rho_p^{(3)}. \label{eq:rho-p}
\end{align}

Putting this together, we can now argue that, for $p\geq 2$, conclusion \eqref{eq:pro3} holds. Thus, choosing $c_p>0$ sufficiently small, we first obtain, by interpolation, that 
$\|u\|_{L^2(\Gamma)}(t)\,\leq\,\frac{D_p}{3C_p}$ is satisfied initially. Next, consider the maximal time interval $[0,t_0]$ such that \eqref{eq:u2-small} holds. We have the following dichotomy: as long as \eqref{eq:decreasing} holds, $Q_p$ is decreasing with time, while if \eqref{eq:decreasing} does {\em not} hold, then by \eqref{eq:rho-p0} and \eqref{eq:rho-p} $Q_p$ is bounded by a modulus of continuity $\rho_p(M)$, and hence by Hölder's inequality $Q_p$ is bounded by a modulus of continuity $\tilde{\sigma}_p(Q_p(0))$. Either way, $Q_p$ remains bounded by a modulus of continuity ${\sigma}_p(Q_p(0))$. Since $p\geq 2$, this also gives, for $c_p>0$ sufficienly small, $\|u\|_{L^2(\Gamma)}(t)\,\leq\,\frac{D_p}{3C_p}$ on $[0,t_0]$, and hence \eqref{eq:u2-small} and \eqref{eq:pro3} must be satisfied for all time. By interpolation, we may also conclude that \eqref{eq:pro3-2} holds for $p\geq 2$.

Now consider $1<p<2$. First, we again use \cite[Section 2 (2.27)]{lady1}, and observe that
\begin{align}
	\int_{\Gamma} u^{2-p}V^p  \leq  C\|u^{2-p}\|_2\int_B|\nabla V^{\frac{p}{2}}|^2 + 2v_{\frac{p}{2}}^2\int_{\Gamma} u^{2-p} \,\leq\, C(\Gamma)\|u\|_2^{2-p} \int_B|\nabla V^{\frac{p}{2}}|^2+ 2v_{\frac{p}{2}}^2\int_{\Gamma} u^{2-p}.\label{eq:alt_uV_est-alt}
\end{align}
Instead of (\ref{eq:alt_Qp_est}), we now have, by (\ref{eq:up-diff2}), (\ref{eq:Vu}), (\ref{eq:alt_uV_est}), and \eqref{eq:alt_uV_est-alt},
\begin{align}
	&\frac{d}{dt}\left(\int_B V^p + c_1\int_{\Gamma} u^p + \int_{\Gamma}u^2\right) \nonumber\\
	\leq & -D_p \int_B |\nabla V^{\frac{p}{2}}|^2 - d_p\int_{\Gamma} |\nabla u^{\frac{p}{2}}|^2 -2d\int_{\Gamma}|\nabla u|^2\nonumber\\
 & +  C_p\left(\|u\|_{L^2(\Gamma)} + \|u\|_{L^2(\Gamma)}^{2-p}\right)\int_B |\nabla V^{\frac{p}{2}}|^2 + C_p'v_{\frac{p}{2}}^2\int_{\Gamma}(u + u^{2-p})\nonumber\\
  \leq & -\left(\frac{D_p}{2} - C_p\left(\|u\|_{L^2(\Gamma)} + \|u\|_{L^2(\Gamma)}^{2-p}\right)\right)\int_B |\nabla V^{\frac{p}{2}}|^2 - 2d\int_{\Gamma}|\nabla u|^2\nonumber\\
 & -  \frac{D_p}{2}\int_B |\nabla V^{\frac{p}{2}}|^2 - d_p\int_{\Gamma}|\nabla u^{\frac{p}{2}}|^2 +  C_p'v_{\frac{p}{2}}^2\int_{\Gamma}(u + u^{2-p}).\label{eq:alt_Qp_est2}
\end{align}

Thus, by the same argument as before, if $\|u\|_2$ is small enough relative to $\frac{D_p}{C_p}$, then either the right-hand side is negative or we have
\begin{align*}
	\int_B V^p \leq v_{\frac{p}{2}}^2\left[2|B| + \frac{2C_pC_p'}{D_p}\int_{\Gamma}(u + u^{2-p})\right],
\end{align*}
which implies (since $p<2$ and $\int_\Gamma u^{2-p}\leq C(\Gamma)\|u\|_1^{2-p}$)
\begin{align*}
	\int_B V^p \leq C_p M^{p}(2|B| + \rho(M)),
\end{align*} 
along with
\begin{align*}
\int_{\Gamma} u^p \leq  2u_{\frac{p}{2}}^2|\Gamma| + C_pv_{\frac{p}{2}}^2\int_{\Gamma}(u + u^{2-p}) \leq  \rho(M),
\end{align*}
by (\ref{u_poinc}), and
\begin{align*}
	\int_{\Gamma}u^2 = \int_\Gamma \left(u-\fint_\Gamma u\right)^2 + |\Gamma|M^2 \leq M^2|\Gamma| + C\int_{\Gamma}|\nabla u|^2\leq M^2|\Gamma| + C\rho\int_{\Gamma}(u + u^{2-p}) \leq \rho(M)
\end{align*}
where $\rho$ always stands for some modulus of continuity, thus giving {\em a priori} control of $Q_p(t)$ by a modulus of continuity if $c_p$ is small enough. This proves \eqref{eq:pro3}, and finally \eqref{eq:pro3-2} follows for $1<p<2$ from \eqref{eq:pro3} and the definition \eqref{eq:def-qp}.
%Control of $Q_2(t)$ in terms of $Q_p(0)$ follows from (\ref{eq:alt_Qp_est}) if $p\geq 2$, and (\ref{eq:alt_Qp_est2}) if $p<2$.
\end{proof}

\section{Some remarks on blow-up in bounded domains}\label{globex3}

We have not been able to obtain global $L^p$-control of solutions to (\ref{V})-(\ref{eq:initial}) for arbitrary data, or indeed to find a counterexample with $L^p$ blow-up in finite time. In particular, decisive vanishing-moment estimates seem to be very difficult to obtain in our model. This being the case, we will now simply make a few observations about qualitative solution behaviour near any possible blow up on a bounded domain, $B$.

Thus, returning to the last term in (\ref{est1}), we see by Young's inequality that
\begin{align}
\int_{\Gamma}V^pu \leq \int_{\Gamma}\frac{1}{p}\frac{u^p}{\eps^p} + \frac{(p-1)}{p}\eps^{p/(p-1)}V^{p^2/(p-1)},
\end{align}
for arbitrary $\eps>0$.

Now set $p=4$, and use the trace inequality
\begin{align}
\|f\|_{L^2(\Gamma)} \leq C\left(\|\nabla f\|_{L^{\frac{3}{2}}(B)} + \|f\|_{L^{\frac{3}{2}}(B)}\right)
\end{align}
applied to $f=V^{\frac{8}{3}}$. Thus,
\bq
\int_{\Gamma} V^{\frac{16}{3}}  & = & \|V^{\frac{8}{3}}\|_{L^2(\Gamma)}^2\nonumber\\
 & \leq & C\left(\left(\int_B|\nabla V^{\frac{8}{3}}|^{\frac{3}{2}}\right)^{\frac{4}{3}}  + \left(  \int_B V^4\right)^{\frac{4}{3}}\right)\nonumber\\
 & = & C\left(\left(\int_B|\nabla V^2|^{\frac{3}{2}}V\right)^{\frac{4}{3}}  + \left(  \int_B V^4\right)^{\frac{4}{3}}\right)\nonumber\\
 &  \leq & C\left(\|V\|_{L^4(B)}^{\frac{4}{3}}\int_B |\nabla V^2|^2 + \left(\int_B V^4\right)^{\frac{4}{3}}\right),
\eq
where we used H\"{o}lder's inequality to get the last line, and hence (\ref{est1}) implies
\begin{eqnarray*}
\frac{d}{dt}\left(\int_{\p B} c_1 u^4  + \int_B V^4\right) & \leq & -3\left(D - C\eps^{4/3}\|V\|_{L^4(B)}^{4/3}\right)\int_B |\nabla V^{2}|^2\nonumber\\
 & + & \frac{C}{\eps^4}\int_{\p B}u^4- 3d\int_{\p B}|\nabla u^{2}|^2\nonumber\\
 & + & C\eps^{4/3}\left(\int_B V^4\right)^{\frac{4}{3}}. %\label{est8}
\end{eqnarray*}

Given this, we make a special choice of $\eps$, namely $\eps^{4/3} = D/(2C\|V\|_{L^4(B)}^{4/3}+1)$, which leads to
\bq
\frac{d}{dt}\left(\int_{\Gamma} c_1 u^4  + \int_B V^4\right) & \leq & -\frac{3}{2}D\int_B |\nabla V^2|^2 - 3d\int_{\Gamma}|\nabla u^2|^2\nonumber\\
 & + & C(\|V\|^4_{L^4(B)}+1)\|u\|^4_{L^4(\Gamma)}
  +   C\left(\int_B V^4\right)^{\frac{4}{3}}\|V\|_{L^4}^{-4/3}\label{eq:b39-1}\\
 & \leq & C\int_B V^4\int_{\Gamma}u^4 + C\int_B V^4. \nonumber%\label{est9}
\eq

Thus, any finite-time $L^4$ blow-up has to occur {\em simultaneously} for $V$ and $u$ (since $Q_4$ can grow at most exponentially if one of $\|u\|_{L^4(\Gamma)}$ and $\|V\|_{L^4(B)}$ remains bounded), and this can be used to show that the trace of $V$ also has to blow up at the same time. Specifically, from (\ref{u}) we get, by Duhamel's principle,
\begin{align}
e^{k_2t}u(y,t) = e^{td\Delta}u(y,0) + \int_0^t e^{(t-s)d\Delta}(k_1e^{k_2s}V(y,s))~ds, \label{eq:40}
\end{align}
and therefore
\begin{align*}
	\limsup_{t~\nearrow~T}\left\|\int_0^t e^{(t-s)d\Delta_{\Gamma}}(k_1e^{k_2s}V(s))~ds\right\|_{L^4(\Gamma)} = \infty,
\end{align*}
where $T$ is the blow-up time. 

Next, we have a smoothing estimate for the heat semi-group on a compact 2-manifold, of the form \cite{TayIII}, Ch.15, Eq (1.15),
\begin{align}
\|e^{td\Delta_{\Gamma}}\|_{\mathcal{L}(W^{r,q}(\Gamma),W^{s,p}(\Gamma))}\leq c t^{-\left(\frac{1}{q} - \frac{1}{p}\right) - (1/2)(s-r)},\label{heat}
\end{align}
for any $p\geq q$ and $s\geq r$ (here $W^{s,p}$ is the $L^p$-type Sobolev space of order $s$).

Hence, setting $p=4$ and $s=0$, there exists a sequence $t_k\rightarrow T^-$ such that $\|V(t_k)\|_{W^{r,q} (\Gamma)}\rightarrow\infty$ as $k\rightarrow\infty$, provided $\frac{1}{q} - \frac{1}{4} - \frac{r}{2}<1$. In particular, $\|V(t_k)\|_{L^1(\Gamma)}$ blows up as $k\rightarrow\infty$.

Note that $L^4$-control of $u$ and $V$ implies global existence of classical solutions, as follows.

Assuming that $Q_4$ is bounded, we have from (\ref{eq:b39-1}), on any interval $[0,T)$ of existence,
\begin{align}
	\int_0^t\|\nabla V^2\|_{L^2(B)}^2 \leq C(1+t), \label{eq:b42-1}
\end{align}
while Duhamel's principle \eqref{eq:40}, and (\ref{heat}) with $p=\infty, r=0, q=8$, yield

\begin{align}
	\|e^{k_2t}u\|_{W^{s,\infty}(\Gamma)}(t) &\leq C + Ce^{k_2t}\int_0^t(t-t')^{-\frac{1}{8}-\frac{s}{2}}\|V\|_{L^8(\Gamma)}(t')~dt' \label{eq:b42-2}
\end{align}
for $s\geq 0$.

Next, the trace inequality \cite{Biez} yields
\begin{align*}
	\|V\|_{L^8(\Gamma)} \,=\, \|V^2\|_{L^4(\Gamma)}^{\frac{1}{2}}\leq C\left(\|\nabla V^2\|_{L^2(B)} + \|V^2\|_{L^2(B)}\right)^{\frac{1}{2}}.
\end{align*}
The right-hand side of this equation is $L^4$-integrable in time, by \eqref{eq:b42-1} and the uniform boundedness of $Q_4$. Thus, we deduce from \eqref{eq:b42-2}, by H\"{o}lder's inequality, that
\begin{align}
	\|e^{k_2t}u(t)\|_{W^{s,\infty}(\Gamma)} &\leq  C + C(T)\|V\|_{L^4(0,T;L^8(\Gamma))}\Big(\int_0^t(t-t')^{-\frac{1}{6}-\frac{2s}{3}}\,dt'\Big)^{\frac{3}{4}}, \label{eq:Wsinfty}
\end{align}
which implies that $u(t)$ stays in $C^{\alpha}(\Gamma)$ 
 for any $\alpha\in [0,\frac{5}{4})$, by Morrey's inequality. This is enough to continue classical solutions indefinitely, by the Schauder theory developed below - see the proof of Theorem \ref{thm2}. 

\begin{remark}
Since $L^4$-boundedness of $u$ or $V$ is enough to guarantee global existence of classical solutions, we have the following: failure of global existence on a bounded domain first implies that either $\|u\|_{L^4(\Gamma)}$ or $\|V\|_{L^4(B)}$ becomes infinite, then that {\em both} $\|u\|_{L^4(\Gamma)}$ and $\|V\|_{L^4(B)}$ blow up, and hence that $\|V\|_{W^{r,q}(\Gamma)}$ has to blow up for any $(r,q)$ satisfying $\frac{1}{q} - \frac{1}{4} - \frac{r}{2}<1$, which in particular yields that the trace of $V$ blows up in $H^{-1}(\Gamma)$ and $W^{-2,4+\epsilon}(\Gamma)$, for any $\eps>0$. We will prove in Theorem \ref{thm2} that smallness of $Q_r(0)$, for any $r>1$, prevents any such blow-up from occurring.
\end{remark}

\section{Existence and uniqueness of local classical solutions}\label{class}

Classical solutions of our model can be constructed with the aid of standard Schauder theory for elliptic and parabolic equations in H\"{o}lder spaces -- one advantage of working in these spaces, as compared to Sobolev spaces, is that $V$ has the same regularity on the boundary, $\Gamma$, as in the bulk, $B$.

Before coming to the existence proof, we first prove a uniqueness result for solutions on bounded domains.\begin{proposition}\label{prop:unique}
Suppose $B\subset\mathbb{R}^3$ is bounded. Then, for a given pair $(V_0,u_0)$ of initial data, there is at most one classical solution $(V,u,c)$ of the system (\ref{V})-(\ref{eq:initial}).
\end{proposition}

\begin{proof}
Suppose that we have two smooth solution triples $(V_i,u_i,c_i)$, $i=1,2$, with the same initial data, on a time interval $[0,T)$.
Then, from (\ref{c}), (\ref{cflux}), we have
\begin{align}
\int_{B}(c_1-c_2)\Delta(c_1-c_2) - \alpha\int_{B}(c_1-c_2)^2=0,
\end{align}
such that integrating by parts and using the Trace Theorem gives
\bq
\int_B|\nabla(c_1-c_2)|^2 + \alpha\int_B (c_1-c_2)^2  & = & \beta\int_{\Gamma}(c_1 - c_2)(u_1 - u_2)\nonumber\\
     & \leq & \beta\|c_1 - c_2\|_{L^2(\Gamma)}\|u_1 - u_2\|_{L^2(\Gamma)}\nonumber\\
     & \leq & C\beta\|c_1 - c_2\|_{H^1(B)}\|u_1 - u_2\|_{L^2(\Gamma)}.
\eq
Dividing through by $\|c_1-c_2\|_{H^1(B)}$, and using $\alpha>0$, we therefore get
\begin{align}
\|c_1 - c_2\|_{H^1(B)}\leq C_{\alpha}\beta\|u_1 - u_2\|_{L^2(\Gamma)}.\label{cdiff}
\end{align}
Next, taking differences in (\ref{V}), applied to $V_1$ and $V_2$, and then testing with $V_1-V_2$  there follows, after an integration by parts and application of the flux condition,
\begin{eqnarray*}
	&&\frac{1}{2}\frac{d}{dt}\|V_1 - V_2\|_{L^2(B)}^2 \\
	& = & -\int_B \nabla(V_1-V_2)\cdot(D\nabla(V_1-V_2) - (V_1\nabla c_1 - V_2\nabla c_2))\nonumber\\
   &  & +\int_{\Gamma} (V_1-V_2)(-k_1(V_1-V_2) + k_2(u_1 - u_2))\nonumber\\
   & = & -D\|\nabla (V_1-V_2)\|_{L^2(B)}^2 + \int_B\nabla(V_1-V_2)((V_1-V_2)\nabla c_1 + V_2\nabla(c_1-c_2))\nonumber\\
   &  & -k_1\|V_1-V_2\|_{L^2(\Gamma)}^2 + k_2\int_{\Gamma}(V_1-V_2)(u_1-u_2)\nonumber\\ 
   & \leq & -D\|\nabla(V_1 -V_2)\|_{L^2(B)}^2 - k_1\|V_1-V_2\|_{L^2(\Gamma)}^2\nonumber\\
   &  & +C\|\nabla(V_1 -V_2)\|_{L^2(B)}\left(\|\nabla c_1\|_{C^0(\bar B)}\|V_1 - V_2\|_{L^2(B)} + C_{\alpha}\beta\|V_2\|_{C^0(\bar B)}\|u_1 -u_2\|_{L^2(\Gamma)}\right)\nonumber\\
   &  & +k_2\|V_1 - V_2\|_{L^2(\Gamma)}\|u_1 - u_2\|_{L^2(\Gamma)},
\end{eqnarray*}
where we used (\ref{cdiff}) to estimate $\|\nabla(c_1-c_2)\|_{L^2(\Gamma)}$.

Hence, since $\|\nabla c_1\|_{C^0(\bar B)}$ and $\|V_2\|_{C^0(\bar B)}$ remain bounded on any time interval $[0,\tilde T]$ with $\tilde T<T$,  Young's inequality yields, for any $t<\tilde T$,
\begin{align}
\frac{d}{dt}\|V_1 - V_2\|_{L^2(B)}^2 \leq C(k_1,k_2,\tilde T)\left(\|V_1 -V_2\|_{L^2(B)}^2 + \|u_1 -u_2\|_{L^2(\Gamma)}^2\right) - \frac{k_1}{2}\|V_1-V_2\|_{L^2(\Gamma)}^2,\label{Vdiff}
\end{align}
while substituting $u_1$, resp. $u_2$, into (\ref{u}), taking differences and testing with $u_1-u_2$ leads to
\begin{align}
\frac{d}{dt}\|u_1 - u_2\|_{L^2(\Gamma)}^2\leq 2k_1\|V_1-V_2\|_{L^2(\Gamma)}\|u_1 - u_2\|_{L^2(\Gamma)} - 2k_2\|u_1 - u_2\|_{L^2(\Gamma)}^2.\label{udiff}
\end{align}
Thus, taking the sum of (\ref{Vdiff}) and (\ref{udiff}), and then using Young's inequality once more to deal with the bad product term, yields
\begin{align}
\frac{d}{dt}\left(\|u_1 - u_2\|_{L^2(\Gamma)}^2 + \|V_1 - V_2\|_{L^2(B)}^2\right)\leq C(k_1,k_2,\tilde T)\left(\|u_1 - u_2\|_{L^2(\Gamma)}^2 + \|V_1 - V_2\|_{L^2(B)}^2\right),
\end{align}
and Gronwall's inequality plus (\ref{cdiff}) shows that $(V_1,u_1,c_1)=(V_2,u_2,c_2)$ on $[0,\tilde T]$ for any $\tilde T<T$.
\end{proof}

With the aid of Schauder's fixed-point theorem, we next prove short-time existence of smooth solutions to (\ref{V})-(\ref{eq:initial}).

\begin{proposition}\label{prop4}
Assume that $B$ and the initial data $(V_0,u_0)$ are of class $C^{2,\gamma}$, $\gamma>0$,
% Let $c_0$ be the solution of
%\begin{align*}
%	\Delta c_0 - \alpha c_0 = 0\quad\mathrm{on}~B\times[0,\infty), \quad \nu\cdot\nabla c_0 = \beta u_0\quad\mathrm{on}~\Gamma\times[0,\infty),
%\end{align*}
and satisfy the compatibility condition
\begin{align*}
	-\nu \cdot D\nabla V_0 \,=\, -\beta u_0V_0 + k_1V_0 -k_2u_0\quad\text{ on }\Gamma.
\end{align*}
Then there exists a unique classical solution $(V,u,c)$ of (\ref{V})-(\ref{eq:initial}) on some time interval $[0,T)$, where $T$ depends only on $\gamma,B$, the model parameters and $\|V_0\|_{C^{2,\gamma}(\bar B)}$, $\|u_0\|_{C^{2,\gamma}(\Gamma)}$. 
\end{proposition}
\begin{proof}
We use an iteration scheme and Schauder's fixed point theorem. First fix an arbitrary $T>0$ and let
\begin{align}
	\mathcal B_K\,:=\, \{ V\in H^{2+\gamma,1+\frac{\gamma}{2}}(B_T)\,:\, V|_{t=0}=V_0,\, |V|_{B_T}^{1+\gamma}\leq K\}. \label{Vbound}
\end{align}
We start the iteration with an arbitrary $V\in\mathcal B_K$, and use this to generate $u:\Gamma_T\rightarrow\mathbb{R}$ via the linear equation
\begin{align}
\p_tu  =  d\Delta_{\Gamma} u + k_1V  - k_2u\quad\text{ in }\Gamma_T,\qquad u(\cdot,0) = u_0\quad\text{ on }\Gamma.  \label{u_k}
\end{align}
Here, solvability in $H^{2+\gamma,1+\frac{\gamma}{2}}(\Gamma_T)$ is guaranteed by Schauder theory \cite[Satz 2.3.24]{Lamm01}.\\
This, in turn, generates $c\in H^{2+\gamma,1+\frac{\gamma}{2}}(B_T)$ via
\begin{align}
\Delta c - \alpha c = 0\quad\mathrm{on}~B_T, \quad \nu\cdot\nabla c = \beta u\quad\mathrm{on}~\Gamma_T,\label{c_k}
\end{align}
(see \cite[Theorem III.3.2]{lady1} for an existence statement).

The iteration procedure is completed by solving 
\bq
	\frac{\p V_{\text{new}}}{\p t} & = & D\Delta V_{\text{new}} - \nabla\cdot(V_{\text{new}}\nabla c)\nonumber\\
  & = & D\Delta V_{\text{new}} - \nabla c\nabla V_{\text{new}} -\alpha c V_{\text{new}}\label{V_new}
\eq
on $B_T$, subject to the boundary condition
\begin{align}
\nu\cdot(D\nabla V_{\text{new}} - V_{\text{new}}\nabla c) = (k_1V_{\text{new}}  - k_2u)\quad\text{ on }\Gamma_T,\label{Vflux_new}
\end{align}
or, equivalently,
\begin{align}
\nu\cdot D\nabla V_{\text{new}} = (k_1V_{\text{new}}  - k_2u) + \beta u V_{\text{new}}\quad\text{ on }\Gamma_T,\label{alt_Vflux}
\end{align}
and the initial condition $V_{\text{new}}(\cdot,0)=V_0$, for the updated function $V_{\text{new}}: B_T\rightarrow\mathbb{R}$. 

Note that unique solvability in $H^{2+\gamma,\frac{2+\gamma}{2}}$ of the linear initial/boundary-value problem for $V_{\text{new}}$ is guaranteed by, for example, \cite[Theorem IV.5.3, p.320]{lady2}, and that we therefore have a well-defined mapping
\begin{align*}
	F:\mathcal B(K)\,\to\, H^{2+\gamma,\frac{2+\gamma}{2}}(B_T)\cap \{V|_{t=0}=V_0\},\quad V\mapsto V_{\text{new}}. 
\end{align*}
A fixed point of this mapping will give the solution we are looking for.

The aim is now to obtain H\"{o}lder estimates for $V_{\text{new}}$ that ensure $V_{\text{new}}\in\mathcal B(K)$, for $T>0$ sufficiently small and $K$ sufficiently large.  

First of all, since the boundary $\Gamma$ is a compact manifold, and since the Comparison Principle gives us control of $\|u\|_{L^\infty(\Gamma_T)}$ in terms of $\|V\|_{L^\infty(B_T)}$, we can stitch together interior-type Schauder estimates (\cite{lady2}, Theorem 10.1, p.351) to get
\begin{align}
	|u|^{(2+\gamma)}_{\Gamma_T}  \leq  C\left(|V|^{(\gamma)}_{B_T} + \|u_0\|_{C^{2,\gamma}(\Gamma)}\right) \leq  C\left(K + \|u_0\|_{C^{2,\gamma}(\Gamma)}\right). \label{ubound}
\end{align}

Moreover, we have the following estimates (\cite[Eq.~3.7, p.137]{lady1}, resp.~\cite[Eq.~7.37, p.349]{TayI})
\begin{align}
	\|c\|_{C^{2,\gamma}(\overline{B})}&\leq C\left(\|c\|_{C^0(\overline{B})} + \|u\|_{C^{1,\gamma}(\Gamma)}\right),\label{cgamma} \\
	\|c\|_{H^2(B)}&\leq C\left(\|c\|_{L^2(B)} + \|u\|_{H^{\frac{1}{2}}(\Gamma)}\right).\label{c_H2}
\end{align}

Now, by the argument of Proposition \ref{prop:unique},  we have $\|c\|_{H^1(B)}\leq\|u\|_{L^2(\Gamma)}$ (cf. (\ref{cdiff})), and so (\ref{c_H2}) implies
\begin{align}
\|c\|_{H^2(B)} \leq C\left(\|c\|_{L^2(B)} + \|u\|_{H^{\frac{1}{2}}(\Gamma)}\right)\leq C\|u\|_{H^{\frac{1}{2}}(\Gamma)},
\end{align}
and hence, by (\ref{cgamma}), with the aid of the Sobolev imbedding $H^2(B)\hookrightarrow C^0(\bar B)$, 
\begin{align}
\|c\|_{C^{2,\gamma}(\overline{B})}\leq C\left(\|u\|_{H^{\frac{1}{2}}(\Gamma)} + \|u\|_{C^{1,\gamma}(\Gamma)}\right)\leq C\|u\|_{C^{1,\gamma}(\Gamma)}.\label{c_k2}
\end{align}

Also, by using the homogeneity of (\ref{c_k}) and taking temporal differences (details given in Appendix \ref{c_k_time_est}), one sees that (\ref{c_k2}) implies
\begin{align}
	\max\left\{|c|^{(\gamma)}_{B_T},|\nabla c|^{(\gamma)}_{B_T}\right\}\leq C|u|^{(2+\gamma)}_{\Gamma_T},\label{cbound}
\end{align}
which by (\ref{ubound}) gives $H^{\gamma,\gamma/2}(B_T)$-control on the coefficients in the right-hand side of (\ref{V_new}). Furthermore, the coefficients and the inhomogeneity in the boundary condition \eqref{alt_Vflux} are by \eqref{ubound} controlled in $H^{1+\gamma,\frac{1+\gamma}{2}}(\Gamma_T)$. With this in hand, we obtain from (\ref{V_new}) and (\ref{alt_Vflux}) by \cite[Theorem 5.3, p.320]{lady2} the following parabolic Schauder estimate for $V_{\text{new}}$:
\begin{align}
	|V_{\text{new}}|^{(2+\gamma)}_{B_T}\leq C\big(\|V_0\|_{C^{2,\gamma}(\overline{B})} + |u|^{(1+\gamma)}_{\Gamma_T}\big),\label{Vschau}
\end{align}
where $C$ depends on $|c|^{(\gamma)}$, $|\nabla c|^{(\gamma)}$ and $|u|^{(1+\gamma)}$.

Combining $|V|_{B_T}^{1+\gamma}\leq K$ with (\ref{ubound}) and (\ref{cbound}) thus results in 
\begin{align}
|V_{\text{new}}|^{(2+\gamma)}\leq C(K,\|u_0\|_{C^{2,\gamma}(\Gamma)})\left(1+\|V_0\|_{C^{2,\gamma}(\overline{B})} \right).\label{Vbound2}
\end{align}
Moreover, by Proposition \ref{KA1_lem} and Remark \ref{KA1_rem} in the Appendix, we have
\begin{align}
|V_{\text{new}}|^{(1+\gamma)}\leq C\left(T^{\delta}|V_{\text{new}}|^{(2+\gamma)} + |V_0|^{(2+\gamma)}\right),\label{Vbound7}
\end{align}
for some $\delta>0$.

Hence,
\begin{align}
|V_{\text{new}}|^{(1+\gamma)}\leq C\left(T^{\delta}C(K,\|u_0\|_{C^{2,\gamma}(\Gamma)})\big(1+\|V_0\|_{C^{2,\gamma}(\overline{B})} \big) + \|V_0\|_{C^{2,\gamma}(\overline{B})}\right),\label{Vbound8}
\end{align}
which entails that if we choose $K$ sufficiently large relative to $\|V_0\|_{C^{2,\gamma}(\overline{B})}$, and $T$ sufficiently small relative to $K$ and $\|u_0\|_{C^{2,\gamma}(\Gamma)}$, then we have
\begin{align}
	|V_{\text{new}}|^{(1+\gamma)}_{B_T}\leq K,\label{Vbound3}
\end{align}
and consequently, using \eqref{Vbound2}, %from (\ref{Vbound2}),
%\begin{align}
%	|V_{\text{new}}|^{(2+\gamma)}_{B_T}\leq C(K),\label{Vbound4}
%\end{align}
%and therefore 
$F:\mathcal{B}(K)\rightarrow \mathcal{B}(K)$ (compactly).

Finally, in order to apply Schauder's fixed-point theorem, we need to show that the mapping $F:\mathcal{B}(K)\rightarrow \mathcal{B}(K)$ is continuous with respect to the $|\cdot|^{(1+\gamma)}$-norm. This is done by applying linear Schauder estimates to differences of (\ref{u_k}), (\ref{c_k}), (\ref{V_new}) and (\ref{Vflux_new}), for starting iterates $V$ and $\overline{V}$ (with the same initial data), which gives
\begin{align}
	|u - \bar{u}|^{(2+\gamma)}\leq C|V - \overline{V}|^{(\gamma)},
\end{align}
\begin{align}
	\max\left\{|c - \bar{c}|^{(\gamma)},|\nabla(c - \bar{c})|^{(\gamma)}\right\}\leq C|u - \bar{u}|^{(2+\gamma)},
\end{align}
and
\begin{align}
	|V_{\text{new}} - \overline{V}_{\text{new}}|^{(2 + \gamma)}\leq C(K)\left(|\nabla(c - \bar{c})|^{(\gamma)} + |u - \bar{u}|^{(1+\gamma)}\right),
\end{align}
as required.

Thus, Schauder's fixed-point theorem can be applied to give us the required short-time solution - uniqueness was already proved above.
\end{proof}

%=============================
%
%=============================
\section{Continuation of classical solutions with small data on bounded domains}\label{globex4}

Here we synthesise the results of Sections \ref{globex} and \ref{class} to obtain a global-in-time classical solution of (\ref{V})-(\ref{eq:initial}) on a bounded domain for small data. More precisely, we have the following theorem.
 
\begin{theorem}\label{thm2}
For bounded $B\subset\mathbb{R}^3$ and initial data $(V_0,u_0)$, all of class $C^{2,\gamma}$, $\gamma >0$, satisfying appropriate compatibility conditions (as in Proposition \ref{prop4}), there exists a unique, global classical solution $(V,u,c)$ of the system (\ref{V})-(\ref{eq:initial}), provided some $Q_p(0)$, $p\in(1,\infty)$, is chosen sufficiently small. 
\end{theorem}

\begin{proof}

Suppose we have a classical solution $(V,u,c)$ on some maximal time interval of existence $[0,T)$, where the existence of such a $T>0$ is guaranteed by Proposition \ref{prop4}.

From Proposition \ref{prop_bdd_dom_Q_p}, we can make $\|u\|_{L^2(\Gamma)}(t)$ arbitrarily small, {\em a priori}, by choosing any $Q_p(0)$ sufficiently small. Thus, in (\ref{eq:alt_Qp_est}), we can set $p=2$ and make $\|u\|_{L^2(\Gamma)}(t)\leq\frac{D_p}{2C_p}$, such that
\begin{align*}	
	\sup_{0<t<T} Q_2(t) + C\left(\int_0^T\|\nabla V\|_{L^2(B)}^2 + \|\nabla u\|_{L^2(\Gamma)}^2\right)\,\leq\, Q_2(0) + C M^3T \,\leq\, \Lambda.
\end{align*}
In the following, we will only need this latter bound.

First we test \eqref{c}, \eqref{cflux} with $c^{p-1}$, and deduce by Gagliardo-Sobolev embedding and interpolation inequalities that
\begin{align*}
	c_p\int_B |\nabla c^{\frac{p}{2}}|^2 + \alpha c^p \,=\, \int_\Gamma \beta u c^{p-1} \,\leq\, \beta\|c^{\frac{p}{2}}\|_{L^4(\Gamma)}^{\frac{2(p-1)}{p}} \|u\|_{L^{\frac{2p}{p+1}}(\Gamma)}\,\leq\, \beta C\|c^{\frac{p}{2}}\|_{H^1(B)}^{\frac{2(p-1)}{p}} \|u\|_{L^2(\Gamma)}^\theta\|u\|_{L^1(\Gamma)}^{1-\theta}
\end{align*}
for some $\theta=\theta(p)\in (0,1)$. This in particular implies
\begin{align}
	\|c^{\frac{p}{2}}\|_{H^1(B)} \,\leq\, C(p,\alpha,B)\beta^{\frac{p}{2}}\|u\|_{L^2(\Gamma)}^{\frac{\theta p}{2}}M^{\frac{(1-\theta)p}{2}} \,\leq\, C(p,\alpha,B)\beta^{\frac{p}{2}} \Lambda^{\frac{p}{4}}. \label{eq:thm7-bdc}
\end{align}
Next Duhamel's principle \eqref{eq:40} and (\ref{heat}) with $s=1-\frac{1}{p}$, $q=4$, $r=0$ yield
\begin{align}
	\|e^{k_2t}u(t)\|_{W^{1-\frac{1}{p},p}(\Gamma)} \leq C + Ce^{k_2t}\int_0^t(t-s)^{-\left(\frac{1}{4}-\frac{1}{p}\right)-\frac{1}{2}\left(1-\frac{1}{p}\right)}\|V\|_{L^4(\Gamma)}(s)~ds. \label{duhamel}
\end{align}
Next, by the trace inequality \cite{Biez}
\begin{align*}
	\|V\|_{L^4(\Gamma)}\leq C\left(\|\nabla V\|_{L^2(B)} + \|V\|_{L^2(B)}\right)
\end{align*}
we deduce
\begin{align}
	\int_0^t\|V\|_{L^4(\Gamma)}^2(s)~ds\leq C\int_0^t\|\nabla V\|_{L^2(B)}^2(s) + \|V\|_{L^2(B)}^2(s)~ds\leq C(\Lambda,T).
\end{align}
Thus, by Young's inequality applied to the integrand on the right-hand side of \eqref{duhamel}, we deduce that $u(t)$ stays in $W^{1-\frac{1}{p},p}(\Gamma)$ for $p<6$ and $t<T$, with
\begin{align*}
	\|u\|_{L^\infty(0,T;W^{1-\frac{1}{p},p}(\Gamma))}\,\leq\, C_p(T,\Lambda)\quad\text{ for any }1\leq p<6.
\end{align*}

By elliptic regularity (\cite{Pruess}, p.57) and \eqref{eq:thm7-bdc} we then obtain %\commentM{Term $\|c\|_2$ missing?}
\begin{align*}
	\|c\|_{W^{2,p}(B)}\leq C_p\Big(\|u\|_{W^{1-\frac{1}{p},p}(\Gamma)} + \|c\|_{L^p(B)}\Big) \,\leq\, C_p\Big(\|u\|_{W^{1-\frac{1}{p},p}(\Gamma)} + \Lambda\Big)\quad\text{ for any }1\leq p<6, % \Big(+ \|c\|_{L^2(B)}\Big),
\end{align*}
and hence taking $p\in (3,6)$, we get by Morrey's inequality that $c\in C^{1+\sigma}(\bar B)$ for all $0<\sigma<\frac{1}{2}$ and that
\begin{align*}
	\|c\|_{L^\infty(0,T;C^{1+\sigma}(\bar B))}\leq C_\sigma\Big(\|u\|_{L^\infty(0,T;W^{1-\frac{1}{p},p}(\Gamma))}+\Lambda\Big)\,\leq\, C_\sigma(T,\Lambda)\quad\text{ for any }0<\sigma<\frac{1}{2}. % \Big(+ \|c\|_{L^2(B)}\Big),
\end{align*}
Thus, all the coefficients in (\ref{V}) and (\ref{Vflux}) are bounded, which entails that $V$ is uniformly bounded on $[0,T)$, by the comparison principle  \cite[Thm 2.3, p.17]{lady2}, and hence, by parabolic regularity \cite[Eq. 1.13, p.316]{TayIII}, 
\begin{align*}
	\|u\|_{\text{Lip}([0,T),C^0(\Gamma))\cap L^{\infty}([0,T),C^{1+\sigma}(\Gamma))}\,\leq\, C_\sigma(T,\Lambda)\quad\text{ for any }0\leq\sigma<1.
\end{align*}
By \cite[Thm 3.2]{Weid02}, this is enough to guarantee 
\begin{align*}
	\|V\|_{W^{2,1}_p(B_T)}\,\leq\, C_p(T,\Lambda)\quad\text{ for all }p\in[1,\infty).
\end{align*}
Hence, by an embedding theorem of Ladyzhenskaya \cite[Lemma II.3.3]{lady2}
\begin{align*}
	\|V\|_{H^{\kappa,\frac{\kappa}{2}}(B_T)}\,\leq\, C_\kappa(T,\Lambda)\quad\text{ for any } 0\leq\kappa<1.
\end{align*}
Thus, by Schauder theory on compact manifolds, 
\begin{align*}
	\|u\|_{H^{2+\kappa,1+\frac{\kappa}{2}}(\Gamma_T)}\,\leq\, C_\kappa(T,\Lambda)\quad\text{ for any } 0\leq\kappa<1,
\end{align*}
and by elliptic Schauder theory
\begin{align*}
	\sup_{t\in (0,T)}\|c\|_{C^{2,\kappa}(\bar B)}(t)\,\leq\, C_\kappa(T,\Lambda)\quad\text{ for any } 0\leq\kappa<1.
\end{align*}
As in the proof of \eqref{cbound} in Proposition \ref{prop4}, we also get control of Hölder seminorms of $c$ by taking temporal differences,
\begin{align*}
	\max\left\{|c|^{(\kappa)}_{B_T},|\nabla c|^{(\kappa)}_{B_T}\right\}\leq C_\kappa(T,\Lambda),\end{align*}
and finally by parabolic Schauder estimates
\begin{align*}
	\|V\|_{H^{2+\kappa,1+\frac{\kappa}{2}}(B_T)} \,\leq\, C_\kappa(T,\Lambda).
\end{align*}
Thus, the local solution guaranteed by Proposition \ref{prop4} can always be continued onto a longer time interval, since a lower bound on the existence time for local-in-time solutions is determined by any $C^{2,\gamma}$-norm of the initial data.
\end{proof}

%=============================
%
%=============================
\section{Global \texorpdfstring{$L^p$}{Lp}-control for arbitrary data in two different regularised models}\label{reg}

We now force global existence of classical solutions for arbitrary data by regularising our model in two different ways. Note in advance that both of our regularised models will preserve positivity of solutions, by the same argument as in Proposition \ref{pos1}.

\subsection{Truncating the flux at the boundary}\label{truncflux}
One way of preventing blow-up of $V$ and $u$ is to regularise by truncating the boundary flux, $q$. Thus, we replace $q=k_1V - k_2u$ with $q_m(k_1V - k_2u)$, taken to be monotonically increasing on physical grounds, where $m\in\R$,
\begin{align}
	q_m\in C^2(\R,\R),\quad q_m(0)=0,\quad -m\leq q_m(\cdot)\leq m. \label{eq:ass-qm}
\end{align}
We therefore immediately have boundedness of the exchange term between bulk and surface. The price we pay, however, is that the boundary condition \eqref{Vflux} and the surface evolution equation \eqref{u} are now nonlinear. Thus, instead of (\ref{u}) we now have
\begin{align}
	\p_tu  =  d\Delta_{\Gamma} u + q_m(k_1 V-k_2 u),\label{u_qm}
\end{align}
and instead of (\ref{Vflux}) we have
\begin{align}
	- \nu\cdot(D\nabla V - V\nabla c) = q_m(k_1V  - k_2u).\label{modVflux}
\end{align}

Since $q_m$ is bounded, we have, by comparison, $u(t)\leq C + mt$, for $t\in[0,T)$. Also note that the boundedness of the flux, $q_m$, implies by parabolic regularity \cite[Eq. 1.13, p.316]{TayIII}, as in the proof of Theorem \ref{thm2},
\begin{align*}
	u\in \text{Lip}([0,T),C(\Gamma))\cap L^{\infty}([0,T),C^{1+\kappa}(\Gamma)),\quad\text{ for all }0\leq\kappa<1,
\end{align*} 
and hence
\begin{align}
	c\in L^{\infty}([0,T),C^{2+\kappa}(\bar B)),\quad\text{ for all }0\leq\kappa<1.  \label{eq:cC0kappa}
\end{align} 
%Next, by taking temporal difference quotients and passing to the limit, it follows that $\partial_tc\in C^0(B_T)\cap L^\infty(0,T;H^1(B))$ is for all $t\in (0,T)$ a weak solution of 
%\begin{align*}
%	-\Delta\partial_tc + \alpha\partial_tc\,=\, 0\quad\text{ in }B,\qquad
%	\nu\cdot \nabla \partial_tc \,=\,\beta\partial_tu\quad\text{ on }\Gamma.
%\end{align*}
%It then follow from \cite[Theorem 3.14]{Nitt11} that
%\begin{align*}
%	\partial_t c\in L^{\infty}(0,T;C^{\sigma}(\bar B)),\quad\text{ for some }0\leq\sigma<1.  %\label{eq:cC0sigma}
%\end{align*} 

The bound \eqref{eq:cC0kappa} and the comparison principle \cite[Thm 2.3, p.17]{lady2} imply that $V$ is uniformly bounded on $[0,T)$. 

By our assumptions \eqref{eq:ass-qm} on $q_m$, and since $V$ is {\em a priori} bounded, it follows from our control on $u$ and \cite[Theorem V.7.1, p.478, and Theorem II.8.1]{lady2}, that some $|V|^{\kappa}_{B_T}$ is also bounded.

Next, using parabolic regularity theory for (\ref{u_qm}) (for example, by covering $\Gamma$ with coordinate patches, pulling the localised equations back to a subset of $\R^2\times(0,T)$ and then applying interior estimates), together with the bound on $|V|^{\delta}_{B_T}$, we get $u\in H^{2+\delta,1+\frac{\delta}{2}}(\Gamma_T)$ for some $\delta>0$. With the aid of Proposition \ref{pro:14}, we also deduce that $c\in H^{1+\delta,\frac{1+\delta}{2}}(\Gamma_T)$.

Thus far we have obtained a-priori bounds
\begin{align}
	|V|^{\delta}_{B_T} + |u|^{2+\delta}(\Gamma_T) + |c|^{1+\delta}(B_T) + \|\partial_t c\|_ {L^{\infty}(0,T;C^{\delta}(\bar B))}\,\leq\, \Lambda \label{eq:apriori-6}
\end{align}
for some $\delta>0$. We next apply \cite[Theorem IV.5.3]{lady2} to \eqref{V}, \eqref{modVflux} and use Ehrling's lemma to deduce
\begin{align*}
	|V|^{2+\delta}_{B_T} \,\leq\, C(\Lambda)(|V|^{1+\delta}(B_T)+1) \,\leq\, C(\Lambda)(\eps|V|^{2+\delta}(B_T)+C_\eps \Lambda +1).
\end{align*}
Choosing $\eps=\eps(\Lambda)$ sufficiently small gives an a-priori bound for $V$ in $ H^{2+\delta,1+\frac{\delta}{2}}(B_T)$.
%The regularity of $u$ allows us to use \cite{lady2} applied to nonlinear Neumann conditions to get that $V$ and $u$ are {\em a priori} in $ H^{2+\eps,1+\frac{\eps}{2}}(B_T)$, resp.~$H^{2+\eps,1+\frac{\eps}{2}}(\Gamma_T)$ \commentM{I guess you refer to LaSU Thm 7.4 p.491? Again this seems to need $\partial_t\nabla c$ uniformly bounded}.
%Thus, any local, smooth solution can always be continued.

To complete this discussion, we claim that the proof of Proposition \ref{prop4} goes through with only slight modifications when $q$ is replaced by a $q_m$ satisfying our assumptions.

Thus, we once again consider for some $T,K>0$ the set
\begin{align*}
	\mathcal B_K\,:=\, \{ V\in H^{2+\gamma,1+\frac{\gamma}{2}}(B_T)\,:\, V|_{t=0}=V_0,\, |V|_{B_T}^{1+\gamma}\leq K\}
\end{align*}
and start our iteration with an arbitrary $V\in\mathcal B_K$. We then solve the initial-value problem
\begin{align}
	\p_tu  =  d\Delta_{\Gamma} u + q_m(k_1V  - k_2u),\quad u(\cdot,0) = u_0  \label{u_k-7}
\end{align}
(\cite[Satz 2.4.5]{Lamm01} implies the existence of a solution $u\in H^{2+\sigma,1+\frac{\sigma}{2}}(\Gamma_{\delta})$ for $0<\sigma<\gamma$ on some maximal interval $[0,\delta)$, $\delta\leq T$, of existence).

As above, we deduce from the boundedness of $q_m$, the regularity of $V$, and Schauder estimates that $u\in H^{2+\gamma,1+\frac{\gamma}{2}}(\Gamma_\delta)$ is bounded a-priori, hence $\delta=T$. This, by  \cite[Theorem III.3.2]{lady1}, generates $c\in H^{2+\gamma,1+\frac{\gamma}{2}}(B_T)$ via
\begin{align*}
	\Delta c - \alpha c = 0\quad\mathrm{on}~B_T, \quad \nu\cdot\nabla c = \beta u\quad\mathrm{on}~\Gamma_T,
\end{align*}
The iteration procedure is completed by solving 
\bq
	\p_t V_{\text{new}} & = & D\Delta V_{\text{new}} - \nabla c\cdot \nabla V_{\text{new}} -\alpha c V_{\text{new}}\quad\text{ on }B\times[0,\infty),\\
	\nu\cdot D\nabla V_{\text{new}} &= &	q_m(k_1V  - k_2u) + \beta u V_{\text{new}} \quad\text{ on }\Gamma,
\eq
where existence of $V_{\text{new}}\in H^{2+\gamma,1+\frac{\gamma}{2}}(B_T)$ is ensured by \cite[Theorem IV.5.3, p.320]{lady2}.

In order to deduce that the map $F$, $V\mapsto V_{\text{new}}$ maps $\mathcal B$ into itself, we redo the estimates from Proposition \ref{prop4}: in the right-hand side of (\ref{ubound}) we have to add $|u|^{(\gamma)}_{\Gamma_T}$, since the flux is now a nonlinear function of $k_1V - k_2u$, but this can be absorbed into the left-hand side, by Proposition \ref{KA1_lem}. Also, the right-hand side of (\ref{Vschau}) picks up an extra term $|V|_{B_T}^{(1+\gamma)}$, and the rest of the proof follows as before.

Summarising, we have thus proved the following.

\begin{theorem}\label{thm1}
Fix $m\in\R$, and assume that $q_m$ satisfies \eqref{eq:ass-qm}. Assume that  $B$ is bounded, that $B$ and the initial data $(V_0,u_0)$ are of class $C^{2,\gamma}$ for some $\gamma >0$, and that the compatibility condition
\begin{align*}
	-\nu \cdot D\nabla V_0 \,=\, -\beta u_0V_0 + q_m(k_1V_0 -k_2u_0)\quad\text{ on }\Gamma
\end{align*}
holds. Then there exists a unique, global classical solution $(V,u,c)$ of the modified model (\ref{V}), (\ref{c}), (\ref{cflux}), (\ref{u_qm}) and (\ref{modVflux}). 
\end{theorem}

\subsection{Truncating the boundary source for  \texorpdfstring{$c$}{c}}

Another reasonable modification of the model is to truncate the boundary flux term for $c$ for large values of $u$. We therefore  consider a function
\begin{align}
	z\in C^3(\R),\quad\text{ bounded and monotonically increasing, with } z(0)=0~\text{and}~|z|\leq z_{\text{max}}, \label{eq:cond-z}
\end{align}
for some $z_{\text{max}}>0$. We then replace the Neumann boundary condition \eqref{cflux} by
\begin{align}
	\nu\cdot\nabla c = z(u)\quad\text{ on }\Gamma. \label{modcflux}
\end{align} 
In this case, we have, for the unfavourable term in the $L^p$-estimate \eqref{nonlinear}
\begin{align*}
	\int_{\Gamma}V^p(\nu\cdot\nabla c)\leq z_{\mathrm{max}}\int_{\Gamma} V^p,
\end{align*}
and therefore we get $L^p$ control of $u$ and $V$ (for arbitrary data and $p\in(1,\infty)$) with the aid of the trace inequality \cite{lady1}
\begin{align}
	\int_{\Gamma}V^p\leq C_{\eps}\int_B V^p + \eps\int_B|\nabla V^{\frac{p}{2}}|^2,\label{trace2}
\end{align}
and
\begin{align*}
	\frac{d}{dt}\left(\int_{\Gamma} c_1 u^p  + \int_B V^p\right)  \leq & -\frac{4(p-1)}{p}D\int_B |\nabla V^{\frac{p}{2}}|^2 - 	\frac{4(p-1)}{p}d\int_{\Gamma}|\nabla u^{\frac{p}{2}}|^2\nonumber\\
   & -  \int_{\Gamma}\frac{p}{(k_1)^{p-1}}(k_1V - k_2u)((k_1V)^{p-1}-(k_2u)^{p-1})\nonumber\\
   & +  z_{\mathrm{max}}\beta (p-1)\int_{\Gamma}V^p, %\label{est11}
\end{align*}
which is the analogue of (\ref{est1}).

Finally, by arguing almost exactly as in Sections \ref{globex3}, \ref{globex4} and \ref{truncflux}, we see that local smooth solutions of the second modified model can always be continued onto a longer time interval. Thus, $L^4$-control implies by Duhamel's principle that $u$ (and hence $z(u)$) stays in $C^{1,\eps}(\Gamma)$ for $\eps<\frac{1}{4}$, cf.~\eqref{eq:Wsinfty}. Thus, $c\in C^{2,\eps}(\bar B)$ (cf.~\eqref{c_k2}), $V$ remains bounded by the comparison principle, and, as before, we eventually see that $u$ and $V$ stay in $H^{2+\gamma,1+\frac{\gamma}{2}}$. The local-in-time theory goes through almost unchanged, since a brief calculation gives
\bq
|z(u)|^{1+\gamma}\leq C|u|^{1+\gamma}\left(1 + |u|^{1+\gamma}\right),\quad |z(u)|^{2+\gamma}\leq C|u|^{2+\gamma}\left(1  + (|u|^{2+\gamma})^2\right),
\eq
and we therefore have the following result.

\begin{theorem}\label{thm3}
Let $z$  as in \eqref{eq:cond-z} be given. Assume that  $B$ is bounded, that $B$ and the initial data $(V_0,u_0)$ are of class $C^{2,\gamma}$ for some $\gamma >0$, and that the compatibility condition
\begin{align*}
	-\nu \cdot D\nabla V_0 \,=\, -\beta u_0V_0 + k_1V_0 -k_2u_0\quad\text{ on }\Gamma
\end{align*}
holds. Then there exists a unique, global classical solution $(V,u,c)$ of the modified model (\ref{V}), (\ref{c}), (\ref{Vflux}), (\ref{u}) and (\ref{modcflux}). 
\end{theorem}

%=============================
%
%=============================

\section{Steady-state analysis}

In this section we look for steady states of our system, that is to say, solutions of
\begin{eqnarray}
	0 & = & D\Delta V - \nabla\cdot(V\nabla c),\label{eq:stat-V}\\
	0  & = & \Delta c - \alpha c\label{eq:stat-c}
\end{eqnarray}
in $B$, subject to the flux conditions
\begin{align}
	- \nu\cdot(D\nabla V - V\nabla c) =  k_1V  - k_2u,\label{eq:stat-V-bdry}
\end{align}
and
\begin{align}
	\nu\cdot\nabla c = \beta u\label{eq:stat-c-bdry}
\end{align} 
on $\Gamma$, such that $u$ satisfies
\begin{align}
	0  =  d\Delta_{\Gamma} u + k_1V  - k_2u,\label{eq:stat-u},
\end{align}
on $\Gamma$.

%\subsection{Existence of steady states for general bounded domains}

\begin{theorem}\label{arb_steady}
Let $B\subset\mathbb{R}^3$ be bounded and of class $C^{2,\gamma}$, $0<\gamma<\frac{1}{2}$. Then, if $\mu>0$ is small enough (only depending on the parameters), there exists a steady-state solution $(V,c,u)$ of (\ref{V})-(\ref{u}) such that $V,c\in C^{2,\gamma}(\bar B)$, $u\in C^{2,\gamma}(\Gamma)$ are all nonnegative, and $\int_\Gamma u\,=\,\mu$.
\end{theorem}

\begin{proof}
We use an iterative scheme, together with Schauder's fixed-point theorem.

First define, for $K, \mu>0$, the closed, convex set
\begin{align*}
U_\mu =\left\{u\in C^{2,\gamma}(\Gamma): \|u\|_{C^{1,\gamma}(\Gamma)}\leq K\right\}\cap\left\{u: u\geq 0,\int_{\Gamma} u = \mu\right\}\subset C^{1,\gamma}(\Gamma),
\end{align*}
which is non-empty if $\mu\leq K|\Gamma|$.

Next, let us start the iteration with an arbitrary $u\in U_\mu$. It follows from \cite[Theorem 3.2, p.137]{lady1} that there exists a unique solution $c\in C^{2,\gamma}(\Gamma)$ of
\begin{align*}
	\Delta c -\alpha c = 0~\mathrm{in}~B,\qquad \nu\cdot\nabla c|_{\Gamma} = \beta u \quad\text{ on }\Gamma.
\end{align*}
By the Schauder estimate \cite[Theorem 3.1, p.135]{lady1} we then have
\begin{align}
	\|c\|_{C^{2,\gamma}(\bar B)}\leq c_1(\alpha,\beta,\Gamma)\left[\beta K + \|c\|_{C^0(\bar B)}\left(c_1 + \alpha + \alpha^{(2+\gamma)/\gamma}\right)\right]. \label{eq:stat-est-c}
\end{align}
%In fact, a more convenient estimate comes from the argument which led to (\ref{c_k2})%, for $\gamma<\frac{1}{2}$ \commentM{check: restriction necessary?}
%. Thus, we have
%\begin{align}
%	\|c\|_{C^{2,\gamma}(\bar B)}\leq c_1(\alpha,\Gamma)\beta\|u\|_{C^{1,\gamma}}(\Gamma)\leq c_1\beta K.
%\end{align}
Moreover, by the maximum principle and $u\geq 0$ we deduce that $c\geq 0$.\\
We next consider for given $u$ and $c$ the elliptic Robin boundary-value problem
\begin{align}
	D\Delta V -\nabla\cdot(V\nabla c) &= 0\quad\text{ in }B, \label{Vstead}\\
	-\nu\cdot(D\nabla V - V\nabla c) &= k_1V - k_2u\quad\text{ on }\Gamma. \label{dVstead}
\end{align}
%Note that a unique solution of (\ref{Vstead})-(\ref{dVstead}) is guaranteed  
By \cite[Theorem 6.31]{GiTr01} and the Fredholm alternative, see the discussion  \cite[p.124]{GiTr01}, a solution $V\in C^{2,\gamma}(\bar B)$ exists as long as the corresponding homogeneous problem, that is \eqref{Vstead}, \eqref{dVstead} with $u$ replaced by zero on the right-hand side, has only the trivial solution. Assuming that $V_*$ is a solution of the homogeneous problem, we deduce that
\begin{align}
	\int_B D|\nabla V_*|^2 +\frac{\alpha}{2} cV_*^2 + k_1 \int_\Gamma V_*^2 \,&=\, \frac{\beta}{2}\int_\Gamma V_*^2u\,\leq\,
	\frac{\beta}{2}\|V_*\|_{L^4(\Gamma)}^2\|u\|_{L^1(\Gamma)}^{\frac{1}{4}}\|u\|_{L^3(\Gamma)}^{\frac{3}{4}} \nonumber\\
	&\leq\, C_S(B)\frac{\beta}{2}\|V_*\|_{H^1(B)}^2\mu^{\frac{1}{4}}K^{\frac{3}{4}}. \label{eq:stat-est1}
\end{align}
On the other hand, the left-hand side can be estimated from below by $c_0(D,k_1,B)\|V_*\|_{H^1(B)}^2$, using the trace theorem in $H^1(B)$. Therefore, the condition
\begin{align}
	\beta\mu^{\frac{1}{4}} K^{\frac{3}{4}}\,\leq\, \frac{c_0(D,k_1,B)}{C_S(B)} \label{eq:cond-muK}
\end{align}
guarantees the existence of a unique solution $V\in C^{2,\gamma}(\bar B)$ of \eqref{Vstead}, \eqref{dVstead}.
Moreover, by \cite[Theorem 6.30]{GiTr01} we have
\begin{align*}
	\|V\|_{C^{2,\gamma}(\bar B)} \,\leq\, C(\gamma,\|c\|_{C^{1,\gamma}(\bar B)},\|u\|_{C^{1,\gamma}(\Gamma)},D,\alpha,\beta)\big(\|V\|_{C^0(\bar B)}+\|u\|_{C^{1,\gamma}(\Gamma)}\big),
\end{align*}
while \cite[Theorem 3.14(iv)]{Nitt11} implies that 
\begin{align*}
	\|V\|_{C^0(\bar B)}\leq C(D,\|\nabla c\|_{C^0(\bar B)},k_1,B)k_2\|u\|_{C^0(\Gamma)}\,\leq\, C(D,\alpha,\beta,B,k_1,k_2,K).
\end{align*}
Putting the last two estimates together and using \eqref{eq:stat-est-c} we deduce
\begin{align}
	\|V\|_{C^{2,\gamma}(\bar B)} \,\leq\, C(\gamma,c_1,D,\alpha,\beta,B,k_1,k_2,K). \label{Vbound5}
\end{align}
Testing \eqref{Vstead} with $V^-\leq 0$ leads to the estimate
\begin{align*}
	\int_B D|\nabla V^-|^2 +\frac{\alpha}{2} c|V^-|^2 + k_1 \int_\Gamma |V^-|^2 \,&=\, \int_\Gamma k_2uV^-+\frac{\beta}{2}|V^-|^2u\,\leq\,
	\frac{\beta}{2}\|V^-\|_{L^4(\Gamma)}^2\|u\|_{L^1(\Gamma)}^{\frac{1}{4}}\|u\|_{L^3(\Gamma)}^{\frac{3}{4}}\\
	&\leq\, C_S(B)\frac{\beta}{2}\|V^-\|_{H^1(B)}^2\mu^{\frac{1}{4}}K^{\frac{3}{4}},
\end{align*}
which is completely analogous to \eqref{eq:stat-est1}. Hence, the condition \eqref{eq:cond-muK} also guarantees nonnegativity of $V$.

Also, from the divergence structure of (\ref{Vstead}), and by (\ref{dVstead}), we have 
\begin{align}
	\int_{\Gamma} V = \frac{k_2}{k_1}\int_{\Gamma}u. \label{mass_const}
\end{align}
%while the comparison principle yields
%
%\begin{align}
%\max_BV\leq 2\frac{k_2}{k_1}\max_{\Gamma}u\leq 2\frac{k_2}{k_1}K,
%\end{align}
%since $c$ is positive, and otherwise the outward normal derivative of $V$ at a point on the boundary where $V$ achieved its maximum would be negative.
%
%Next, applying the Schauder estimate of Ladyzhenskaya (\cite{lady1}, p.135-136, Theorem 3.1), in conjunction with the above estimates on $c$ and $\max V$, to $V$ gives
%\bq
%\|V\|_{C^{2,\gamma}(\bar B)} & \leq & c_2(B,D)\Big\{k_2K + 2\frac{k_2K}{k_1}\Big(\alpha c_1\beta K + (k_1 + \beta K + \alpha c_1\beta K)^{\frac{(2+\gamma)}{\gamma}}\Big.\Big.\nonumber\\
%& & + (k_1 +\beta K + c_1 +  c_1\beta K)^{\frac{2+\gamma}{1+\gamma}} + (c_1 + k_1 + \beta K + c_1\beta K)^{2+\gamma} + c_1 \Big.\Big.\Big)\Big\}.\label{Vbound5}
%\eq

To complete the iteration, we finally solve
\begin{align}
	d\Delta_{\Gamma}\un +k_1V - k_2\un =0,\label{u_new}
\end{align}
for $\un=\un(V)$ on $\Gamma$, which implies, via (\ref{mass_const}),
\begin{align}
	\int_{\Gamma} \un = \int_{\Gamma} u,
\end{align}
thus preserving the mass constraint.

By the maximum principle, we deduce that $\|\un\|_{C^0(\Gamma)}\leq \frac{k_1}{k_2}\|V\|_{C^0(\bar B)}$. Hence, by Schauder theory and \eqref{Vbound5}
\begin{align}
	\|\un\|_{C^{2,\gamma}(\Gamma)}\leq C(\gamma,c_1,D,\alpha,\beta,B,k_1,k_2,K). \label{u_new4}  
\end{align}
Let us finally control the $C^{1,\gamma}(\Gamma)$-norm of $\un$. By elliptic $L^p$-theory we obtain that for any $2< p<4$
\begin{align*}
	\|\un\|_{W^{2,p}(\Gamma)}\,\leq\, C(d,k_1,\Gamma)\|V\|_{L^p(\Gamma)}\,\leq\, C(d,k_1,\Gamma)\|V\|_{L^1(\Gamma)}^{1-\theta}\|V\|_{L^4(\Gamma)}^\theta 
\end{align*}
for some $\theta=\theta_p\in(0,1)$. Since $W^{2,p}(\Gamma)$ embeds continuously into $C^{1,\gamma}(\Gamma)$ for $\gamma=1-\frac{2}{p}$, and since
$H^1(B)$ embeds continuously into $L^4(\Gamma)$ we therefore obtain, using \eqref{mass_const}, for any $0<\gamma<\frac{1}{2}$ that
\begin{align}
	\|\un\|_{C^{1,\gamma}(\Gamma)} \,\leq\, C(d,k_1,B,\gamma)k_2^{1-\theta_\gamma}\mu^{1-\theta_\gamma}\|V\|_{H^1(B)}^{\theta_\gamma} \label{eq:stat-H1}
\end{align}
for some $0<\theta_\gamma<1$. Testing \eqref{Vstead}, \eqref{dVstead} with $V$ yields
\begin{align*}
	\int_B D|\nabla V|^2 +\frac{\alpha}{2} cV^2 + k_1 \int_\Gamma V^2 \,&=\,\int_\Gamma\left( k_2uV+ \frac{\beta}{2}V^2u\right)\\
	&\leq\, k_2\|u\|_{L^{\frac{4}{3}}(\Gamma)}\|V\|_{L^4(\Gamma)} +
	\frac{\beta}{2}\|V\|_{L^4(\Gamma)}^2\|u\|_{L^1(\Gamma)}^{\frac{1}{4}}\|u\|_{L^3(\Gamma)}^{\frac{3}{4}}\\
	&\leq\, k_2\|u\|_{L^1(\Gamma)}^{\frac{1}{2}}\|u\|_{L^2(\Gamma)}^{\frac{1}{2}}\|V\|_{H^1(B)} + C_S(B)\frac{\beta}{2}\|V\|_{H^1(B)}^2\mu^{\frac{1}{4}}K^{\frac{3}{4}},\\
	&\leq\, k_2\mu^{\frac{1}{2}}K^{\frac{1}{2}}\|V\|_{H^1(B)} + C_S(B)\frac{\beta}{2}\|V\|_{H^1(B)}^2\mu^{\frac{1}{4}}K^{\frac{3}{4}},
\end{align*}
and the condition \eqref{eq:cond-muK} ensures that
\begin{align*}
	\|V\|_{H^1(B)} \,\leq\, C(D,k_1,B,C_S(B))k_2\mu^{\frac{1}{2}}K^{\frac{1}{2}}.
\end{align*}
Using this in \eqref{eq:stat-H1} finally gives
\begin{align*}
	\|\un\|_{C^{1,\gamma}(\Gamma)} \,\leq\, C(d,k_1,B,\gamma,D,C_S(B))k_2\mu^{1-\frac{\theta_\gamma}{2}}K^{\frac{\theta_\gamma}{2}}.
 \end{align*}

Thus, in order to satisfy \eqref{eq:cond-muK} and, to guarantee  $\|\un\|_{C^{1,\gamma}(\Gamma)} \leq K$, we need in addition
\begin{align}
	 C(d,k_1,B,\gamma,D,C_S(B))k_2\mu^{1-\frac{\theta_\gamma}{2}}\,\leq\, K^{1-\frac{\theta_\gamma}{2}}. \label{eq:cond-muK2}
\end{align}
We finally need $U_\mu$ to be non-empty. For this the condition
\begin{align}
	\mu\,=\, \int_\Gamma u \,\leq\, |\Gamma|\|u\|_{C^0(\Gamma)}\,\leq\, |\Gamma| K \label{eq:cond-muK2a}
\end{align}
is necessary and sufficient, hence both \eqref{eq:cond-muK2}, \eqref{eq:cond-muK2a} are satisfied if
\begin{align}
	C'(d,k_1,B,\gamma,D,C_S(B),k_2)\mu\,\leq\, K. \label{eq:cond-muK2b}
\end{align}
We now check that it is possible to meet all conditions if $0<\mu<\mu_0$, where $\mu_0>0$ satisfies
\begin{align*}
	\mu_0 \,\leq\, \frac{c_0(D,k_1,B)}{\beta C_S(B)}C'(d,k_1,B,\gamma,D,C_S(B),k_2)^{-\frac{3}{4}}.
%	C(d,k_1,B,\gamma,D,C_S(B))^6k_2^{12-6\theta_\gamma}\mu_0^{10-6\theta_\gamma}\,&\leq\  \frac{c_0(D,k_1,B)^4}{C_S(B)^4\beta^4},\\
%	\mu_0^4 \,&\leq\, |\Gamma|^3  \frac{c_0(D,k_1,B)^4}{C_S(B)^4\beta^4}.
\end{align*}

Consequently, for $\mu<\mu_0$ and $K$ chosen such that \eqref{eq:cond-muK}, \eqref{eq:cond-muK2b} hold, the mapping $T:u\mapsto\un$ takes $U_\mu$ into itself. Moreover, by (\ref{u_new4}), $T(U_\mu)$ is compact in the $C^{1,\gamma}$-topology. Continuity of $T$ follows easily by taking differences, $T(u_1) - T(u_2)$, and applying Schauder estimates to the $c, V$ and $u$ equations in turn.
Hence, all the conditions of Schauder's fixed-point theorem are satisfied, and we are therefore guaranteed a solution of the steady-state problem.
\end{proof}

%\subsection{Spherically-symmetric steady states on a ball}

For the particular case of spherical cell shapes, we can prove existence of stationary states for any given total mass.
\begin{proposition}
Let  $B\subset\mathbb{R}^3$ be a ball of radius $R$. Then for every $M>0$ there exists a spherically symmetric steady state $(V,c,u)$ with total mass  $\int_B V + \int_\Gamma u \,=\,M$. Moreover, $u$ is constant and $V,c$ are smooth.
\end{proposition}
\begin{proof}
The symmetry immediately implies that $u$ is constant: $u(y) = u_0$, and also $V = V_0 = \frac{k_2}{k_1}u_0$ on $\Gamma$ -  it remains to determine $u_0$, $V=V(r)$ and $c=c(r)$ in the interior, where $r$ is the radius.

Equation \eqref{eq:stat-c} reads, in spherical polar coordinates,
\begin{align*}
	rc'' + 2c'  =  \alpha cr,
\end{align*}
and hence $c$ is a modified (spherical) Bessel function of the first kind, namely $c=c_0\frac{\sinh(\sqrt{\alpha}r)}{\sqrt{\alpha}r}$.
The boundary condition \eqref{eq:stat-c-bdry} fixes $c_0$ as a function of $u_0$:
\begin{align*}
	c_0 = \beta u_0\left(\frac{\cosh(\sqrt{\alpha}R)}{R} - \frac{\sinh(\sqrt{\alpha}R)}{\sqrt{\alpha}R^2}\right)^{-1}.
\end{align*}
With this we can find $V$ by solving \eqref{eq:stat-V}, \eqref{eq:stat-V-bdry}, which gives
\begin{align*}
	V = V_0e^{(c(r) - c(R)/D}.
\end{align*}

The final condition is then  
\begin{align*}
	M = 4\pi R^2u_0 + 4\pi\int_0^Rr^2\frac{k_2}{k_1}u_0e^{(c(r)-c(R))/D}~dr.
\end{align*} 
Here, the right-hand is zero when $u_0=0$, and tends to infinity as $u_0\rightarrow\infty$. Hence, there exists at least one spherically symmetric steady-state solution for any given mass. 
%For $\frac{k_2}{k_1}$ small we have $M(u_0)$ monotonically increasing, which gives uniqueness. Whether or not this property holds in general is surprisingly difficult to determine.
\end{proof}

\section{Conclusions}
In this paper, we have analysed the well-posedness of a model for spontaneous cell polarisation. The model takes the form of a rather generic coupled bulk/surface reaction-diffusion-advection system with feedback and  total-mass conservation which could conceivably be of physical relevance beyond its original biological motivation. Coupling between bulk and surface pde's has recently gained a lot of attention, %\cite{BaFL14,MaCV15,GKRR15,ElRV15,MoSh15}, 
but well-posedness results for models with advection seem to be much thinner on the ground. Our bulk equations resemble the PKS model for chemotaxis, but with a somewhat weaker form of feedback, which nevertheless leaves open the (unproven) possibility of finite-time blow-up.

We have made several contributions to developing an existence theory for the model. First, we proved, under a precise smallness condition on the initial data, that classical solutions exist globally in time. This is in some sense analogous to one half of the well-established dichotomy in the PKS model. The critical space for the smallness condition is $L^2$ for the surface variable, $u$, and between $L^1$ and any $L^p$, $p>1$ for the bulk variable, $V$. Scaling arguments \cite{Vela16} seem to suggest that these findings are, with respect to the usual $L^p$ spaces, optimal, but we are not able to prove an analogue of the other half of the PKS dichotomy (blow-up for concentrated data), due to the absence of suitable moment-estimates.

The main model we have considered uses a rather ad-hoc choice for the constitutive functions that determine the boundary coupling. One could easily argue for the presence of some kind of saturation effect that prevents the associated Neumann sources from becoming infinite. We have therefore analysed corresponding regularisations of our model, and have proved that this leads unconditionally  to global-in-time existence of classical solutions. This precludes any, seemingly unphysical, occurrence of singular mass-concentration. Nevertheless, blow-up solutions of the original system can be expected to have approximate counterparts in the regularised systems, with strongly heterogeneous concentration profiles that may still be interpreted as polarised states.

Our partial results leave open a number of avenues for further inquiry. In terms of the biological motivation, a thorough analysis of the qualitative behaviour (polarisation, more complex pattern formation), possibly in combination with numerical simulations would be the natural next step. Again, the complex structure of the model makes this a rather challenging prospect, and we therefore leave this question for the future.
  
\begin{appendix}
 
\section{Positivity of solutions}
Here we show that classical solutions corresponding to nonnegative initial data remain nonnegative on their interval of existence.
\begin{proposition}\label{pos1}
Suppose $B\subset\mathbb{R}^3$ is a smooth bounded domain and $(V,c,u)$ is a classical solution of (\ref{V})-(\ref{u}) on ${B}\times[0,T)$ corresponding to non-negative initial data $V_0$, $u_0$. Then $V(\cdot,t)$, $c(\cdot,t)$ and $u(\cdot,t)$ are also non-negative for all $t\in [0,T)$.
\end{proposition}

\begin{proof}
We test (\ref{V}) with the negative part of $V$, i.e.~$V^-:=-\min\{V,0\}\geq 0$, and (\ref{u}) with the negative part of $u$, i.e. $u^-$. Using Stampacchia's Lemma, and arguing as in the proof of Lemma \ref{lem:p-powers}, we arrive at 
\bq
\frac{d}{dt}\left(\|V^-\|_{L^2(B)}^2 + c_1\|u^-\|_{L^2(\Gamma)}^2\right) & \leq & C\int_{\Gamma}(k_1V^-- k_2u^-)(k_1V-k_2u)\nonumber\\
 & - & 2D\|\nabla V^-\|_{L^2(B)}^2 + \beta\int_{\Gamma}(V^-)^2u -\alpha\int_B (V^-)^2c.\label{eq:b98-1}
\eq
Next observe that
\begin{align}
	\sup_{0\leq t\leq \tilde T}\big(\|u\|_{L^\infty(\Gamma)}(t) + \|c\|_{L^\infty(B)}(t)\big)\leq C(\tilde T)\quad\text{ for any }0<\tilde T<T \label{eq:A-Ttilde}
\end{align}
and that $(k_1V^- - k_2u^-)(k_1V-k_2u)\leq 0$ by definition. Using (\ref{trace2}) we therefore deduce that
\begin{align}
	  \beta\int_{\Gamma}(V^-)^2u  \,\leq\, & C_{\eps}\beta\|u\|_{L^\infty(\Gamma)}\|V^-\|_{L^2(B)}^2 +  \eps\beta\|u\|_{L^\infty(\Gamma)}\|\nabla V^-\|_{L^2(B)}^2, \notag\\
	 \leq \,& C(D,\beta,\tilde T)\|V^-\|_{L^2(B)}^2 + 2D\|\nabla V^-\|_{L^2(B)}^2, \label{uV_minus2}
\end{align}
for all $t\leq\tilde T$, hence by \eqref{eq:b98-1} and \eqref{eq:A-Ttilde}
\begin{align*}
	\frac{d}{dt}\left(\|V^-\|_{L^2(B)}^2 + c_1\|u^-\|_{L^2(\Gamma)}^2\right) & \leq  C(D,\beta,\tilde T)\|V^-\|_{L^2(B)}^2 +\alpha C(\tilde T)\|V^-\|_{L^2(B)}^2,
\end{align*}
and by Gronwall's inequality we deduce that $V^-(\cdot,t)=0$ and $u^-(\cdot,t)=0$ for all $t\leq\tilde T<T$, hence for all $0<t<T$.

Furthermore, testing (\ref{c}) with $c^-$ results in
\begin{align*}
	\|\nabla c^-\|_{L^2(B)}^2 + \alpha\|c^-\|_{L^2(B)}^2 \,= - \, \beta\int_{\Gamma}c^- u \,\leq\,0,
\end{align*}
which ensures that also $c^-= 0$, as required.
\end{proof}

We can also deduce a similar result for the case $B=\R^3_+$. % (the upper half-space, our prototypical model of an unbounded domain with boundary).
\begin{proposition}\label{pos2}
Suppose $B=\mathbb{R}^3_+$ and $(V,c,u)$ is a classical solution of (\ref{V})-(\ref{u}) on $\overline{B}\times[0,T)$ corresponding to non-negative initial data $V_0$, $u_0$. Assume further that for all $0<\tilde T<T$ there exists $C(\tilde T)>0$ such that for all $0\leq t\leq \tilde T$
\begin{align*}
	&\quad\|u(t)\|_{L^1(\Gamma)\cap L^\infty(\Gamma)}+\|V(t)\|_{L^1(B)\cap L^\infty(B)} +\\
	+ &\|\nabla u(t)\|_{L^2(\Gamma)} + \|\nabla V(t)\|_{L^2(B)} + \|c(t)\|_{L^\infty(B)}+ \|\nabla c(t)\|_{L^2(B)} \,\leq\, C(\tilde T).
\end{align*}
Then $V(\cdot,t)$, $c(\cdot,t)$ and $u(\cdot,t)$ are also non-negative for all $t\in[0,T)$.
\end{proposition}

\begin{proof}
As above, we find that \eqref{eq:b98-1}, \eqref{eq:A-Ttilde} hold. Using the trace inequality \cite[Section V.3, Example 4]{Lion84} and Young's inequality, we further obtain
\begin{align*}
	\int_\Gamma (V^-)^2u \,\leq\, \|V^-\|_{L^3(\Gamma)}^2\|u\|_{L^3(\Gamma)} \,&\leq\, C(\tilde T)\|\nabla V^-\|_{L^2(B)}^{\frac{5}{3}}\|V^-\|_{L^2(B)}^{\frac{1}{3}} \\
	&\leq\, \frac{2D}{\beta}\|\nabla V^-\|_{L^2(B)}^2 + C(D,\beta,\tilde T)\|V^-\|_{L^2(B)}^2,
\end{align*}
%Hence we deduce
%\begin{align*}
%	\frac{d}{dt}\left(\|V^-\|_{L^2(B)}^2 + c_1\|u^-\|_{L^2(\Gamma)}^2\right) & \leq  C(D,\beta,\tilde T)\|V^-\|_{L^2(B)}^2 +\alpha C(\tilde T)\|V^-\|_{L^2(B)}^2,
%\end{align*}
and the rest of the argument is identical to that for bounded $B$.
\end{proof}

\section{Some useful parabolic-H\"{o}lder estimates}\label{c_k_time_est}

\begin{proposition}\label{pro:14}
Suppose $B\subset\mathbb{R}^3$ is a $C^{2,\gamma}$ domain, for some $0<\gamma<1$. %For any $T>0$, let $B_T = B\times [0,T]$, $\Gamma = \Gamma$ and $\Gamma_T = \Gamma\times [0,T]$. 
If $u$  belongs to the parabolic  H\"{o}lder space $H^{2+\gamma,\frac{1}{2}(2+\gamma)}(\Gamma_T)$, and $c:B_T\rightarrow\mathbb{R}$ is determined by (\ref{c}) and (\ref{cflux}), then
\begin{align}
\max\left\{|c|^{(\gamma)}_{B_T},|\nabla c|^{(\gamma)}_{B_T}\right\}\leq C|u|^{(2+\gamma)}_{\Gamma_T},\label{cbound_again}
\end{align}
for some $C>0$.
\end{proposition}

\begin{proof}
First we write out the relevant parabolic H\"{o}lder norms explicitly \cite{lady2}:
\begin{align}
|c|_{B_T}^{(\gamma)} = \langle c\rangle_x^{(\gamma)} + \langle c\rangle_t^{(\frac{\gamma}{2})} + \|c\|_{C^0(B_T)},
\end{align}
where, for example, $\langle c\rangle_x^{(\gamma)}$ is the H\"older constant of $c$, of order $\gamma$, with respect to $x$, on $B_T$.

Similarly, 
\begin{align}
|\nabla c|_{B_T}^{(\gamma)} = \langle \nabla c\rangle_x^{(\gamma)} + \langle \nabla c\rangle_t^{(\frac{\gamma}{2})} + \|\nabla c\|_{C^0(B_T)},
\end{align}
while for $u$ we have
\bq
|u|_{\Gamma_T}^{(2+\gamma)} & = & \|u\|_{C(\Gamma_T)} + \|\nabla u\|_{C(\Gamma_T)} + \|\nabla^2 u\|_{C(\Gamma_T)} + \|u_t\| _{C(\Gamma_T)}\nonumber\\
 & + & \langle\nabla^2 u\rangle_x^{(\gamma)} + \langle u_t\rangle_x^{(\gamma)} + \langle u_t\rangle_t^{(\frac{\gamma}{2})}\nonumber\\
 & + & \langle \nabla u\rangle_t^{(\frac{\gamma + 1}{2})} + \langle \nabla^2 u\rangle_t^{(\frac{\gamma}{2})}.
\eq

Thus, the $C^0$-norms of $c$ and $\nabla c$ are trivially dominated by $|u|_{\Gamma_T}^{(2+\gamma)}$, as a consequence of (\ref{c_k2}).

Next, by the homogeneity of (\ref{c}), (\ref{cflux}), we can apply an estimate of the form (\ref{c_k2}) to temporal differences, which gives, for fixed $t, t'$,
\begin{align*}\left\|\frac{c(\cdot,t) - c(\cdot,t')}{(t-t')^{\frac{\gamma}{2}}}\right\|_{C^{2,\gamma}(\bar B)}\leq C \left\|  \frac{u(\cdot,t) - u(\cdot,t')}{(t-t')^{\frac{\gamma}{2}}}\right\|_{C^{1,\gamma}(\Gamma)}.\end{align*}
Hence,
\bq\sup_x\left|\frac{c(x,t) - c(x,t')}{(t-t')^{\frac{\gamma}{2}}}\right| & \leq & C \left\|  \frac{u(\cdot,t) - u(\cdot,t')}{(t-t')^{\frac{\gamma}{2}}}\right\|_{C^2(\Gamma)}\nonumber\\
 & = & C\left[\sup_y\left|\frac{u(y,t) - u(y,t')}{(t-t')^{\frac{\gamma}{2}}}\right|\right.\nonumber\\
 & + & \sup_y\left|\frac{\nabla u(y,t) - \nabla u(y,t')}{(t-t')^{\frac{\gamma}{2}}}\right|\nonumber\\
 & + & \left.\sup_y\left|\frac{\nabla^2 u(y,t) - \nabla^2 u(y,t')}{(t-t')^{\frac{\gamma}{2}}}\right|~\right].
\eq
Taking the supremum over $t,t'\in[0,T]$ gives $\langle c\rangle_t^{(\frac{\gamma}{2})}\leq C|u|_{\Gamma_T}^{(2+\gamma)}$, while (\ref{c_k2}) trivially gives $\langle c\rangle_x^{(\gamma)}\leq C\sup_t\|u(\cdot,t)\|_{C^2(\Gamma)}\leq C|u|_{\Gamma_T}^{(2+\gamma)}$.

Finally, since
\begin{align*}\left\|\frac{\nabla c(\cdot,t) - \nabla c(\cdot,t')}{(t-t')^{\frac{\gamma}{2}}}\right\|_{C^{1,\gamma}(\bar B)}\leq\left\|\frac{c(\cdot,t) - c(\cdot,t')}{(t-t')^{\frac{\gamma}{2}}}\right\|_{C^{2,\gamma}(\bar B)}\leq C \left\|  \frac{u(\cdot,t) - u(\cdot,t')}{(t-t')^{\frac{\gamma}{2}}}\right\|_{C^2(\Gamma)},\end{align*}
repeating the same argument as above finishes the proof.
\end{proof}

\begin{proposition}\textbf{[\cite{KA1}, Lemma B.1]}\label{KA1_lem} For a function $f\in H^{2+l,\frac{1}{2}(2+l)}(\Omega_T)$, $0<l<1$, such that $\Omega_T=\bar\Omega\times[0,T]$,  $\Omega\subset\mathbb{R}$, we have
\begin{align}
|f|^{(l+1)}_{\Omega_T}\leq C(\Omega,l)\left(T^\delta|f|^{(l+2)}_{\Omega_T} + \|f(\cdot,0)\|_{C^2(\bar\Omega)}\right),
\end{align}
where $\delta=\min\{l/2,(1-l)/2\}$.
\end{proposition}
\begin{remark}\label{KA1_rem}
This lemma also works for any reasonable (Lipschitz, say) bounded domain $\Omega\subset\mathbb{R}^n$, since in the proof one merely needs the Lipschitz constant of a differentiable function $f$ to be controlled by $\|\nabla f\|_{C^0(\Omega)}$, which is indeed true on Lipschitz domains. 
\end{remark}

\end{appendix}

%==========================================
% 
%==========================================
%\bibliographystyle{plain}
%\bibliography{chemo}
%%\begin{thebibliography}{100}
%%\end{thebibliography}
\def\cprime{$'$}

\end{document}